\documentclass[a4paper,12pt]{article}

\usepackage[left=2cm,right=2cm, top=2cm,bottom=3cm,bindingoffset=0cm]{geometry}

\usepackage{verbatim}
\usepackage{amsmath}
\usepackage{amsthm}
\usepackage{amssymb}
\usepackage{delarray}
\usepackage{cite}
\usepackage{hyperref}
\usepackage{mathrsfs}
\usepackage{tikz}
\usetikzlibrary{patterns}
\usepackage{caption}
\DeclareCaptionLabelSeparator{dot}{. }
\newcommand{\al}{\alpha}

\newcommand{\ga}{\gamma}
\newcommand{\de}{\delta}
\newcommand{\la}{\lambda}
\newcommand{\om}{\omega}

\newcommand{\eps}{\varepsilon}

\newcommand{\iy}{\infty}

\theoremstyle{plain}

\numberwithin{equation}{section}

\newtheorem{thm}{Theorem}[section]
\newtheorem{lem}[thm]{Lemma}

\newtheorem{cor}[thm]{Corollary}

\theoremstyle{definition}

\newtheorem{alg}[thm]{Algorithm}
\newtheorem{ip}[thm]{Inverse Problem}

\theoremstyle{remark}

\newtheorem{remark}[thm]{Remark}

\DeclareMathOperator*{\Res}{Res}

 \allowdisplaybreaks

\begin{document}

\begin{center}
{\large\bf Solvability and stability of the inverse problem for the quadratic differential pencil}
\\[0.2cm]
{\bf Natalia P. Bondarenko, Andrey V. Gaidel} \\[0.2cm]
\end{center}

\vspace{0.5cm}

{\bf Abstract.} The inverse spectral problem for the second-order differential pencil with quadratic dependence on the spectral parameter is studied. We obtain sufficient conditions for the global solvability of the inverse problem, prove its local solvability and stability. The problem is considered in the general case of complex-valued pencil coefficients and arbitrary eigenvalue multiplicities. Studying local solvability and stability, we take the possible splitting of multiple eigenvalues under a small perturbation of the spectrum into account. Our approach is constructive. It is based on the reduction of the nonlinear inverse problem to a linear equation in the Banach space of infinite sequences. The theoretical results are illustrated by numerical examples.
  
\medskip


\medskip


\section{Introduction} \label{sec:intr}

Consider the boundary value problem $L(q_0, q_1)$ in the form
\begin{gather} \label{eqv}
    -y'' + (2 \la q_1(x) + q_0(x)) y = \la^2 y, \quad x \in (0, \pi), \\ \label{bc}
    y(0) = y(\pi) = 0,
\end{gather}
where $\la$ is the spectral parameter, $q_j$ are complex-valued functions, called \textit{the potentials}, $q_j \in W_2^{j-1}(0, \pi)$, $j = 0, 1$, that is, $q_1 \in L_2(0, \pi)$, $q_0 = \sigma'$, where $\sigma \in L_2(0, \pi)$ and the derivative is understood in the sense of distributions. Note that the class $W_2^{-1}$ contains, in particular, the Dirac delta-functions and the Coulomb-type singularities $\frac{1}{x}$, which are widely used in quantum mechanics \cite{AGHH05}. Equation~\eqref{eqv} can be rewritten in the following equivalent form:
\begin{gather*}
\ell(y) + 2 \la q_1(x) y = \la^2 y, \\
\ell(y) := -(y^{[1]})' - \sigma(x) y^{[1]} - \sigma^2(x) y,
\end{gather*}
where $y^{[1]} := y' - \sigma y$ is the quasi-derivative, $y, y^{[1]} \in AC[0, \pi]$, $\ell(y) \in L_2(0, \pi)$. Clearly, the eigenvalue problem~\eqref{eqv}-\eqref{bc} generates the pencil of second-order differential operators with quadratic dependence on the spectral parameter.

The paper is concerned with the theory of inverse spectral problems, which consist in recovery of operators from their spectral characteristics. The most complete results in inverse problem theory have been obtained for the Sturm-Liouville equation~\eqref{eqv} with $q_1(x) \equiv 0$ (see the monographs \cite{Mar77, Lev84, PT87, FY01} and references therein). In particular, Sturm-Liouville inverse problems with singular potentials of class $W_2^{-1}$ were studied in \cite{HM03, SS05} and other papers. Investigation of inverse problems for differential pencils induced by equation~\eqref{eqv} with nonlinear dependence on the spectral parameter causes principal difficulties comparing with the classical Sturm-Liouville problems. Therefore, a number of open questions still remain in this direction. At the same time, inverse problems for equation~\eqref{eqv} are used in various applications, e.g., for modeling interactions between colliding relativistic particles in quantum mechanics \cite{JJ72} and for studying vibrations of mechanical systems in viscous media \cite{Yam90}.

For the quadratic differential pencil \eqref{eqv} on a finite interval with the regular potentials $q_j \in W_2^j[0, \pi]$, $j = 0, 1$, and the Robin boundary conditions $y'(0) - h y(0) = 0$, $y'(\pi) + H y(\pi) = 0$, the solvability conditions for the inverse spectral problem were obtained by Gasymov and Guseinov~\cite{GG81}. Later on, their approach was applied for investigation of inverse problems for the pencils with non-separated boundary conditions \cite{Gus86, GN00, GN07}. Hryniv and Pronska \cite{HP12, Pron13, Pron13-2, Pron13-3} developed an approach to inverse problems for the pencils of form~\eqref{eqv}-\eqref{bc} with  the singular potentials $q_j \in W_2^{j-1}[0, \pi]$, $j = 0, 1$. Their approach is based on the reduction of equation~\eqref{eqv} to a first-order system. In the recent paper \cite{HM20}, the analogous reduction was applied to the inverse scattering problem for the quadratic differential pencil on the half-line. However, the results of the mentioned papers have the common disadvantage that consists in the requirement of real-valued potentials and positivity of some operator. Under this requirement, the eigenvalues of the pencil are real and simple, which makes the situation similar to the classical Sturm-Liouville operators and significantly simplifies the investigation of inverse problems. However, in the general case, the pencil \eqref{eqv}-\eqref{bc} can have multiple and/or non-real eigenvalues even if the potentials $q_j$ are real-valued.

Buterin and Yurko \cite{BY06, BY12} developed another approach, which allows to solve inverse problems for quadratic differential pencils with the complex-valued potentials $q_j \in W_2^j[0, \pi]$, $j = 0, 1$, and without any additional restrictions. The approach of \cite{BY06, BY12} is based on the method of spectral mappings \cite{FY01, Yur02}. This method allows to reduce a nonlinear inverse spectral problem to a linear equation in an appropriate Banach space, by using contour integration in the $\la$-plane and the theory of analytic functions. The approach based on the method of spectral mappings was also applied to the pencils of the matrix Sturm-Liouville operators \cite{Bond16, BF14}, to the scalar pencils on the half-line \cite{Yur00}, to the half inverse problem \cite{But11}, and to the pencils on graphs (see \cite{Yur17, Yur19} and references therein). However, the results obtained by using this approach for differential pencils include only uniqueness theorems and constructive procedures for solving inverse problems. The most principal questions of solvability and stability for the general case of complex-valued potentials were open. The present paper aims to fill this gap.

It is also worth mentioning that, in recent years, a number of scholars have been actively studying inverse problems for quadratic differential pencils (see \cite{AED21, GKG18, GP18, GWY19, SP14, Yang13, YY14, WHS19, WSW20} and other papers of these authors). The majority of those results are concerned with partial inverse problems, inverse nodal problems, and recovery of the pencils from the interior spectral data.

The aim of this paper is to study solvability and stability of the inverse spectral problem for the pencil~\eqref{eqv}-\eqref{bc}. Developing the ideas of the method of spectral mappings \cite{FY01, BY06, BY12, Yur02}, we reduce the inverse problem to the so-called main equation, which is a linear equation in the Banach space of bounded infinite sequences. The most important difficulties of our investigation are related with eigenvalue multiplicities. The multiplicities influence on the definition of the spectral data and on the construction of the main equation. Moreover, under a small perturbation of the spectrum, multiple eigenvalues can split into smaller groups, which complicates the analysis of local solvability and stability. Nevertheless, we take this effect into account and obtain the results valid for arbitrary multiplicities. For dealing with multiple eigenvalues, we use some ideas previously developed for the non-self-adjoint Sturm-Liouville operators in \cite{But07, BSY13, Bond20}.

Thus, the following main results are obtained.

\begin{itemize}
    \item In Section~\ref{sec:alg}, the spectral data of the quadratic differential pencil are defined and a constructive solution of the inverse problem is obtained for the case of the complex-valued singular potentials $q_j \in W_2^{j-1}[0, \pi]$, $j = 0, 1$ (see Algorithm~\ref{alg:1}). Note that the results of \cite{BY06, BY12} are limited to the case of regular potentials $q_j \in W_2^j[0, \pi]$, $j = 0, 1$. Our constructive procedure implies Theorem~\ref{thm:uniq} on the uniqueness of the inverse problem solution.
    \item In Section~\ref{sec:approx}, we construct infinitely differentiable approximations $q_j^N$ of the potentials $q_j$, by using finite spectral data (Theorem~\ref{thm:approx}). This theorem plays an auxiliary role in the further sections, but also can be treated as a separate result.
    \item In Section~\ref{sec:sol}, we prove Theorem~\ref{thm:solve}, which provides sufficient conditions for the global solvability of the inverse problem. Theorem~\ref{thm:solve} implies Corollary~\ref{cor:locsimp} on the local solvability and stability of the inverse problem for spectrum perturbations that do not change eigenvalue multiplicities.
    \item In Section~\ref{sec:mult}, we prove Theorem~\ref{thm:locmult} and Corollary~\ref{cor:disc} on the local solvability and stability of the inverse problem in the general case, taking splitting of multiple eigenvalues into account.
    \item In Section~\ref{sec:ex}, our theoretical results are illustrated by numerical examples. Namely, we approximate a pencil having a double eigenvalue by a family of pencils with simple eigenvalues.
\end{itemize}

\section{Constructive solution} \label{sec:alg}

In this section, we define the spectral data of the problem $L(q_0, q_1)$ and develop Algorithm~\ref{alg:1} for recovery of the potentials $q_j \in W_2^{j-1}(0, \pi)$, $j = 0, 1$, from the spectral data. The nonlinear inverse problem is reduced to the linear equation~\eqref{main}, which plays a crucial role in the constructive solution and also in study of solvability and stability for the inverse problem. In addition, relying on Algorithm~\ref{alg:1}, we obtain the uniqueness of the inverse problem solution (Theorem~\ref{thm:uniq}). We follow the strategy of \cite{BY12}, so some formulas and propositions of this section are provided without proofs. However, it is worth mentioning that our constructive solution is novel for the case of the singular potentials $q_j \in W_2^{j-1}(0, \pi)$, $j = 0, 1$. The most important difference from the regular case $q_j \in W_2^j[0, \pi]$, $j = 0, 1$, is the construction of the regularized series \eqref{defeps1}-\eqref{defeps3} and the analysis of their convergence in Lemma~\ref{lem:conv}. The results of this section will be used in the further sections for investigation of solvability and stability issues.

Let us start with preliminaries. Denote by $S(x, \la)$ the solution of equation~\eqref{eqv} satisfying the initial conditions $S(0, \la) = 0$, $S^{[1]}(0, \la) = 1$. Put $Q(x) := \int_0^x q_1(t) \, dt$. The results of \cite{Pron13-3} yield the following lemma.

\begin{lem} \label{lem:trans}
There exist such functions $\mathscr K(x, t)$ and $\mathscr N(x, t)$ that
\begin{align*}
S(x, \la) & = \frac{\sin (\la x - Q(x))}{\la} + \frac{1}{\la} \int_{-x}^x \mathscr K(x, t) \exp(i \la t) \, dt, \\
S^{[1]}(x, \la) & = \cos (\la x - Q(x)) + \int_{-x}^x \mathscr N(x, t) \exp(i \la t) \, dt,
\end{align*}
$\mathscr K(x, .)$ and $\mathscr N(x, .)$ belong to $L_2(0, x)$ for each fixed $x \in (0, \pi]$. Moreover, the norms $\| \mathscr K(x, .) \|_{L_2(0, x)}$ and $\| \mathscr N(x, .) \|_{L_2(0, x)}$ are bounded uniformly with respect to $x \in (0, \pi]$.
\end{lem}

Put $\mathbb Z_0 := \mathbb Z \setminus \{ 0 \}$, $\om_0 = \frac{1}{\pi} Q(\pi)$. The problem $L(q_0, q_1)$ has a countable set of the eigenvalues $\{ \la_n \}_{n \in \mathbb Z_0}$ (counted with multiplicities), which coincide with the zeros of the analytic characteristic function $\Delta(\la) := S(\pi, \la)$ and have the following asymptotics (see \cite{Pron13-3}):
\begin{equation} \label{asymptla1}
    \la_n = n + \om_0 + \varkappa_n, \quad n \in \mathbb Z_0.
\end{equation}
Here and below, the same notation $\{ \varkappa_n \}$ is used for various sequences from $l_2$.

Introduce the notations
$$
\mathbb S := \{ n \in \mathbb Z_0 \colon \forall k < n \:\: \la_k \ne \la_n \}, \quad m_n := \# \{ k \in \mathbb Z_0 \colon \la_k = \la_n \},
$$
that is, $\{ \la_n \}_{n \in \mathbb S}$ is the set of all the distinct eigenvalues and $m_n$ is the multiplicity of the eigenvalue $\la_n$. 
Without loss of generality, we impose the following assumption.

\medskip

\textbf{Assumption} $(\mathcal O)$: $\la_n \ne \la_k$ for $n k < 0$ and $\la_n = \la_{n + 1} = \dots = \la_{n + m_n - 1}$, $n \in \mathbb S$.

\medskip

Together with the eigenvalues, we use additional spectral characteristics for reconstruction of the pencil. Let us define two types of such characteristics. Denote
$$
S_{\nu}(x, \la) = \frac{1}{\nu!} \frac{d^{\nu}}{d\la^{\nu}} S(x, \la),
\qquad S_{n + \nu}(x) := S_{\nu}(x, \la_n), \quad n \in \mathbb S, \quad \nu = \overline{0, m_n-1}.
$$
Put $S_{\nu}(x, \la) = 0$ for $\nu < 0$. Define \textit{the generalized weight numbers} as follows:
\begin{multline} \label{defal}
\al_{n + \nu} := \int_0^{\pi} (2 (\la_n - q_1(x)) S_{m_n - 1}(x, \la_n) - S_{m_n-2}(x, \la_n)) S_{\nu}(x, \la_n) \, dx \\ + \int_0^{\pi} S_{m_n-1}(x, \la_n) S_{\nu - 1}(x, \la_n) \, dx, \quad \nu = \overline{0, m_n-1}, \quad n \in \mathbb S. 
\end{multline}
Note that $\{ \al_n \}_{n \in \mathbb Z_0}$ generalize the classical weight numbers of the self-adjoint Sturm-Liouville operator (see, e.g, \cite{Mar77, FY01}).

We call \textit{the Weyl solution} the function $\Phi(x, \la)$ satisfying equation~\eqref{eqv} and the boundary conditions $\Phi(0, \la) = 1$, $\Phi(\pi, \la) = 0$. \textit{The Weyl function} is defined as $M(\la) := \Phi^{[1]}(0, \la)$. The function $M(\la)$ is meromorphic with the poles $\{ \la_n \}_{n \in \mathbb S}$ of the corresponding multiplicities $\{ m_n \}_{n \in \mathbb S}$. Denote
$$
M_{n + \nu} := \Res_{\la = \la_n} (\la - \la_n)^{\nu} M(\la), \quad n \in \mathbb S, \quad \nu = \overline{0, m_n-1}.
$$
Note that the generalized weight numbers $\{ \al_n \}_{n \in \mathbb Z_0}$ and the coefficients $\{ M_n \}_{n \in \mathbb Z_0}$ determine each other uniquely by the formula (see \cite{BY12}):
\begin{equation} \label{relalM}
\sum_{j = 0}^{\nu} \al_{n + \nu - j} M_{n + m_n - j - 1} = -\de_{\nu, 0}, \quad n \in \mathbb S, \quad \nu = \overline{0, m_n-1}, 
\end{equation}
where $\de_{\nu,0}$ is the Kronecker delta. Therefore, the following two inverse problems are equivalent:

\begin{ip} \label{ip:weight}
Given $\{ \la_n, \al_n \}_{n \in \mathbb Z_0}$, find $q_0$, $q_1$.
\end{ip}

\begin{ip} \label{ip:M}
Given $\{ \la_n, M_n \}_{n \in \mathbb Z_0}$, find $q_0$, $q_1$.
\end{ip}

Further, we focus on Inverse Problem~\ref{ip:M}. For convenience, let us call the collection $\{ \la_n, M_n \}_{n \in \mathbb Z_0}$ \textit{the spectral data} of the problem $L$.

One can easily obtain the asymptotics
\begin{equation} \label{asymptM1}
    M_n = -\frac{n}{\pi} (1 + \varkappa_n), \quad n \in \mathbb Z_0.
\end{equation}

In the regular case $q_j \in W_2^j[0, \pi]$, the asymptotics \eqref{asymptla1} and \eqref{asymptM1} can be improved (see \cite{BY12}):
\begin{align} \label{asymptla2}
    & \la_n = n + \om_0 + \frac{\om_1}{\pi n} + \frac{\varkappa_n}{n}, \\ \label{asymptM2}
    & M_n = -\frac{n}{\pi} \left( 1 + \frac{\om_0 - \om_2}{n} + \frac{\varkappa_n}{n}\right),
\end{align}
where 
$$
\om_1 = \frac{1}{2} \int_0^{\pi} (q_0(x) + q_1^2(x)) \, dx, \quad \om_2 = q_1(0).
$$

Note that the function $\sigma = \int q_1(x) \, dx$ is determined by $q_1$ uniquely up to an additive constant. However, this constant does not influence on the definitions of the Weyl function and the spectral data. Thus, in the regular case, we may assume that $\sigma(x) = \int_0^x q_1(t) \, dt$, so $\sigma(0) = 0$ and $y^{[1]}(0) = y'(0)$.

Along with $L = L(q_0, q_1)$, we consider another problem $\tilde L = L(\tilde q_0, \tilde q_1)$ of the same form but with different coefficients $\tilde q_j \in W_2^{j-1}(0, \pi)$, $j = 0, 1$. We agree that, if a symbol $\ga$ denotes an object related to $L$, then the symbol $\tilde \ga$ with tilde will denote the similar object related to $\tilde L$. Note that the quasi-derivatives for these two problems are supposed to be different: $y^{[1]} = y' - \sigma y$ for $L$ and $y^{[1]} = y' - \tilde \sigma y$ for $\tilde L$. Without loss of generality, we suppose that the both eigenvalue sequences $\{ \la_n \}_{n \in \mathbb Z_0}$ and $\{ \tilde \la_n \}_{n \in \mathbb Z_0}$ satisfy Assumption $(\mathcal O)$. 

Introduce the notations
\begin{gather} \nonumber
    \hat Q := Q - \tilde Q, \quad \Theta(x) := \cos \hat Q(x), \quad \Lambda(x) := \sin \hat Q(x), \\ \nonumber
    \la_{n,0} := \la_n, \quad \la_{n,1} := \tilde \la_n, \quad M_{n, 0} := M_n, \quad M_{n,1} := \tilde M_n, \\ \nonumber \mathbb S_0 := \mathbb S, \quad \mathbb S_1 := \tilde {\mathbb S}, \quad
    m_{n,0} := m_n, \quad m_{n,1} := \tilde m_n, \\ \nonumber
    S_{k + \nu,i}(x) := S_{\nu}(x, \la_{k,i}), \quad
    \tilde S_{k + \nu, i}(x) := \tilde S_{\nu}(x, \la_{k,i}), \quad k \in \mathbb S_i, \quad \nu = \overline{0, m_{k,i}-1}, \quad i = 0, 1, \\ \label{defD}
    \tilde D(x, \la, \mu) := \frac{S(x, \la) S'(x, \mu) - S'(x, \la) S(x, \mu)}{\la - \mu} = \int_0^x (\la + \mu - 2q_1(t)) \tilde S(t, \la) \tilde S(t, \mu) \, dt, \\ \label{defA}
    \tilde A_{n + \nu, i}(x, \la) := \sum_{p = \nu}^{m_{n,i}-1} \frac{1}{(p-\nu)!} M_{n + p, i} \frac{\partial^{p - \nu}}{\partial \mu^{p-\nu}} \tilde D(x, \la, \mu)\Big|_{\mu = \la_{n,i}}, \\ \label{defP}
    \tilde P_{n + \nu, i; k,j}(x) := \frac{1}{\nu!} \frac{\partial^{\nu} }{\partial \la^{\nu}} \tilde A_{k,j}(x, \la) \Big|_{\la = \la_{n,i}}, \quad n \in \mathbb S_i, \quad \nu = \overline{0, m_{n,i}-1}, \quad i = 0, 1.
\end{gather}

By using the contour integration in the $\la$-plane, Buterin and Yurko~\cite{BY12} have derived the following relation:
\begin{equation} \label{rel-cont}
    \Theta(x) \tilde S_{n,i}(x) = S_{n,i}(x) - \sum_{k \in \mathbb Z_0} (\tilde P_{n,i; k,0}(x) S_{k,0}(x) - \tilde P_{n,i; k,1}(x) S_{k,1}(x)), \quad n \in \mathbb Z_0, \: i = 0, 1.
\end{equation}
However, it is inconvenient to use \eqref{rel-cont} as the main equation of the inverse problem, since the series converges only ``with brackets''. Therefore, below we transform \eqref{rel-cont} into an equation in the Banach space of infinite sequences.

Define the numbers
\begin{equation} \label{defchi}
\theta_n := |\la_n - \tilde \la_n|, \quad 
\chi_n := \left\{\begin{array}{ll}
            \theta_n^{-1}, \quad & \theta_n \ne 0, \\
            0, \quad & \theta_n = 0.
          \end{array}\right.
\end{equation}
Let $J$ be the set of indices $(n, i)$, $n \in \mathbb Z_0$, $i = 0, 1$. For each fixed $x \in [0, \pi]$, define the sequence $\phi(x) := [\phi_{n,i}(x)]_{(n, i) \in J}$ of the elements
\begin{equation} \label{defphi}
\begin{bmatrix}
\phi_{n,0}(x) \\ \phi_{n,1}(x)
\end{bmatrix} = n
\begin{bmatrix}
\chi_n & -\chi_n \\ 0 & 1 
\end{bmatrix}
\begin{bmatrix}
S_{n,0}(x) \\ S_{n,1}(x)
\end{bmatrix}.
\end{equation}
Analogously define $\tilde \phi(x) = [\tilde \phi_{n,i}(x)]_{(n, i) \in J}$, replacing $S_{n,i}$ by $\tilde S_{n,i}$. It is clear that, for each fixed $x \in [0, \pi]$, the sequences $\phi(x)$ and $\tilde \phi(x)$ belong to the Banach space $\mathfrak B$ of bounded sequences $a = [a_{n,i}]_{(n, i) \in J}$ with the norm $\| a \|_{\mathfrak B} = \sup_{(n, i) \in J} |a_{n,i}|$.

Define the elements $\tilde H_{n,i; k,j}(x)$ for $(n, i), (k, j) \in J$ by the formula
\begin{equation} \label{defH}
\begin{bmatrix}
\tilde H_{n,0; k,0}(x) & \tilde H_{n,0; k,1}(x) \\
\tilde H_{n,1; k,0}(x) & \tilde H_{n,1; k,1}(x) 
\end{bmatrix} = 
\frac{n}{k} 
\begin{bmatrix}
\chi_n & -\chi_n \\ 0 & 1 
\end{bmatrix}
\begin{bmatrix}
\tilde P_{n,0; k,0}(x) & \tilde P_{n,0; k,1}(x) \\
\tilde P_{n,1; k,0}(x) & \tilde P_{n,1; k,1}(x) 
\end{bmatrix}
\begin{bmatrix}
\theta_k & 1 \\ 0 & -1 
\end{bmatrix}.
\end{equation}
Consider the linear operator $\tilde H(x) \colon \mathfrak B \to \mathfrak B$, $\tilde H = \tilde H(\{ \la_n, M_n\}_{n \in \mathbb Z_0}, \tilde L)$, acting as follows:
$$
\tilde H(x) a = \sum_{(k,j) \in J} \tilde H_{n,i; k,j}(x) a_{k,j}, \quad a = [a_{k,j}] \in \mathfrak B.
$$

Define the numbers $\{ \xi_n \}_{n \in \mathbb Z_0}$ as follows:
\begin{gather} \label{defxi}
    \xi_{k + \nu} := |\la_k - \tilde \la_k| + \frac{1}{k} \sum_{p = \nu}^{m_k - 1} |M_{k + p} - \tilde M_{k + p}|, \quad k \in \mathbb S \cap \tilde {\mathbb S}, \quad m_k = \tilde m_k, \quad \nu = \overline{0, m_k-1}, \\ \nonumber
    \xi_n := 1 \quad \text{for the rest of $n$.}
\end{gather}
Suppose that $\om_0 = \tilde \om_0$. Then, it follows from \eqref{asymptla1} and \eqref{asymptM1} that $\{ \xi_n \} \in l_2$. Using the standard technique based on Schwarz's lemma (see \cite[Section~1.6.1]{FY01} and \cite[Section~4]{BY12}), we obtain the estimate
\begin{equation} \label{estH}
    |\tilde H_{n,i; k,j}(x)| \le C \xi_k \left(\frac{1}{|n - k| + 1} + \frac{1}{|k|} \right),
\end{equation}
where $(n, i), (k, j) \in J$, $x \in [0, \pi]$. Here and below, the same symbol $C$ denotes various positive constants independent of $n$, $i$, $k$, $j$, $x$, etc. Consequently, 
\begin{multline*}
\| \tilde H(x) \|_{\mathfrak B \to \mathfrak B} \le C \sup_{n \in \mathbb Z_0} \sum_{k \in \mathbb Z_0} \xi_k \left(\frac{1}{|n - k| + 1} + \frac{1}{|k|} \right) \\ 
\le  C \sqrt{\sum_{k \in \mathbb Z_0} \xi_k^2} \left( \sup_{n \in \mathbb Z_0} \sqrt{\sum_{k \in \mathbb Z_0} \frac{1}{(|n - k| + 1)^2}} + \sqrt{\sum_{k \in \mathbb Z_0} \frac{1}{|k|^2}} \right) < \iy,
\end{multline*}
that is, for each fixed $x \in [0, \pi]$, the operator $\tilde H(x)$ is bounded in $\mathfrak B$. 

In the case $\tilde q_j \in W_2^j[0, \pi]$, $j = 0, 1$, the functions $\tilde \phi_{n,i}(x)$ and $\tilde H_{n,i; k,j}(x)$ belong to $C^1[0, \pi]$ and, for $(n, i)$, $(k, j) \in J$, $x  \in [0, \pi]$,
\begin{equation} \label{estHp}
    |\tilde \phi'_{n,i}(x)| \le C |n|, \quad |\tilde H'_{n,i; k,j}(x)| \le C |n| \xi_k.
\end{equation}

Due to the introduced notations, relation~\eqref{rel-cont} yields the following theorem.

\begin{thm} \label{thm:maineq}
For each fixed $x \in [0, \pi]$, the following relation holds:
\begin{equation} \label{rel-B}
\Theta(x) \tilde \phi(x) = (I - \tilde H(x)) \phi(x),
\end{equation}
where $I$ is the identity operator in $\mathfrak B$.
\end{thm}

Assume that $\Theta(x) \ne 0$, $x \in [0, \pi]$. Denote
$$
z(x) = [z_{n,i}(x)]_{(n,i) \in J} := \frac{\phi(x)}{\Theta(x)}.
$$
Then \eqref{rel-B} implies the following equation in $\mathfrak B$ with respect to $z(x)$ for each fixed $x \in [0, \pi]$:
\begin{equation} \label{main}
\tilde \phi(x) = (I - \tilde H(x)) z(x).
\end{equation}
We call equation~\eqref{main} \textit{the main equation} of the inverse problem. The solvability of the main equation is given by the following theorem, which is proved similarly to Theorem~4.3 from \cite{BY12}.

\begin{thm} \label{thm:inv}
If $\Theta(x) \ne 0$, then the operator $(I - \tilde H(x))$ has a bounded inverse in $\mathfrak B$, so the main equation~\eqref{main} is uniquely solvable.
\end{thm}

Using the solution of the main equation, one can construct the solution of Inverse Problem~\ref{ip:M}. For this purpose, we introduce the functions
\begin{gather} \label{defvn}
    v_{n,0}(x) := \frac{1}{n} (\theta_n z_{n,0}(x) + z_{n,1}(x)), \quad
    v_{n,1}(x) := \frac{1}{n} z_{n,1}(x), \quad n \in \mathbb Z_0, \\ \nonumber
    \tilde B_{n + \nu, i}(x) := \sum_{p = \nu}^{m_{n,i} - 1} M_{n + p, i} \tilde S_{n + p - \nu, i}(x), \quad n \in \mathbb S_i, \quad \nu = \overline{0, m_{n,i}-1}, \quad i = 0, 1.    
\end{gather}
Let $n_0 \in \mathbb N \cup \{ 0 \}$ be the smallest index such that $m_{n,i} = 1$ for all $|n| > n_0$, $i = 0, 1$. Consider the series
\begin{align} \label{defeps1}
    \eps_1(x) & := \sum_{(k, j) \in J} (-1)^j \tilde B_{k,j}(x) v_{k,j}(x),\\ \label{defeps2}
    \eps_2(x) & := \sum_{\substack{|k| \le n_0 \\ j = 0, 1}} (-1)^j \la_{k,j} \tilde B_{k,j}(x) v_{k,j}(x) + \sum_{\substack{|k| > n_0 \\ j = 0, 1}} (-1)^j M_{k,j} \left( \la_{k,j} \tilde S_{k, j}(x) v_{k,j}(x) - \frac{1}{2 \la_{k,j}}\right), \\ \label{defeps3}
    \eps_3(x) & := \sum_{\substack{|k| \le n_0 \\ j = 0, 1}} (-1)^j \tilde B'_{k,j}(x) v_{k,j}(x) + \sum_{\substack{|k| > n_0 \\ j = 0, 1}} (-1)^j M_{k,j} \left( \tilde S'_{k,j}(x) v_{k,j}(x) + \frac{\Lambda(x)}{2 \la_{k,j} \Theta(x)}\right), \\ \label{defeps4}
    \eps_{4}(x) & := \sum_{j = 0}^{1} (-1)^j \sum_{m_{k,j} > 1} \sum_{\nu = 0}^{m_{k,j} - 2} \tilde B_{k + \nu + 1, j}(x) v_{k + \nu, j}(x).
\end{align}

\begin{lem} \label{lem:conv}
$\eps_1 \in W_2^1[0, \pi]$, $\eps_2, \eps_3 \in L_2(0, \pi)$.
\end{lem}

\begin{proof}
Obviously, it is sufficient to prove the lemma for $n_0 = 0$. Note that $v_{k,j}(x) = (\Theta(x))^{-1} S_{k,j}(x)$ and $\tilde B_{k,j}(x) = M_{k,j} \tilde S_{k,j}(x)$.

\smallskip

\textbf{Step 1.} Using~\eqref{defeps1}, we derive
\begin{multline} \label{smeps1}
\eps_1(x) \Theta(x) = \sum_{k \in \mathbb Z_0} (M_{k,0} \tilde S_{k,0}(x) S_{k,0}(x) - M_{k,1} \tilde S_{k,1}(x) S_{k,1}(x)) 
= \sum_{k \in \mathbb Z_0} (M_{k,0} - M_{k,1}) \tilde S_{k,0}(x) S_{k,0}(x) \\ + \sum_{k \in \mathbb Z_0} M_{k,1} (\tilde S_{k,0}(x) - \tilde S_{k,1}(x)) S_{k,0}(x) + \sum_{k \in \mathbb Z_0} M_{k,1} \tilde S_{k,1}(x) (S_{k,0}(x) - S_{k,1}(x)).
\end{multline}
It follows from \eqref{asymptla1}, \eqref{asymptM1}, \eqref{defxi}, and Lemma~\ref{lem:trans} for $S(x, \la)$ and $\tilde S(x, \la)$ that
\begin{gather*}
    |M_{k,i}| \le C |k|, \quad |M_{k,0} - M_{k,1}| \le C |k| \xi_k, \quad
    |S_{k,i}(x)|, |\tilde S_{k,i}(x)| \le C |k|^{-1}, \\ 
    |S_{k,0}(x) - S_{k,1}(x)| \le C |k|^{-1} \xi_k, \quad
    |\tilde S_{k,0}(x) - \tilde S_{k,1}(x)| \le C |k|^{-1} \xi_k
\end{gather*}
for $k \in \mathbb Z_0$, $x \in [0, \pi]$. Hence
$$
|\eps_1(x) \Theta(x)| \le C \sum_{k \in \mathbb Z_0} |k|^{-1} \xi_k < \iy,
$$
that is, the series \eqref{smeps1} converges absolutely and uniformly with respect to $x \in [0, \pi]$. Since $\Theta \in W_2^1[0, \pi]$, $\Theta(x) \ne 0$, $x \in [0, \pi]$, this yields $\eps_1 \in C[0, \pi]$.

\smallskip

\textbf{Step 2.} Differentiating \eqref{defeps1} and using the relations $S' = S^{[1]} + \sigma S$, $\tilde S' = \tilde S^{[1]} + \tilde \sigma \tilde S$, we obtain
\begin{align*}
    \eps_1'(x) & = (\Theta(x))^{-1} Z_1(x) + Z_2(x), \\
    Z_1(x) & := \sum_{(k, j) \in J} (-1)^j M_{k,j} (\tilde S^{[1]}_{k,j}(x) S_{k,j}(x) + \tilde S_{k,j}(x) S^{[1]}_{k,j}(x)), \\
    Z_2(x) & := \left(\sigma(x) + \tilde \sigma(x) - \frac{\Theta'(x)}{\Theta(x)}\right) \eps_1(x).
\end{align*}
Obviously, $Z_2 \in L_2(0, \pi)$. Let us prove the same for $Z_1$. Lemma~\ref{lem:trans} yields
$$
\tilde S_{k,j}^{[1]}(x) S_{k,j}(x) + \tilde S_{k,j}(x) S_{k,j}^{[1]}(x) = \frac{\sin (2 \la_{k,j} x - Q(x) - \tilde Q(x))}{\la_{k,j}} + \frac{\varkappa_{k,j}(x)}{\la_{k,j}},
$$
where $\{ \varkappa_{k,j}(x) \}$ is some sequence satisfying
\begin{equation} \label{sumkappa}
\sum_{(k, j) \in J} |\varkappa_{k,j}(x)|^2 \le C, \quad 
\sum_{k \in \mathbb Z_0} |\varkappa_{k,0}(x) - \varkappa_{k,1}(x)| \le C
\end{equation}
uniformly with respect to $x \in [0, \pi]$. Consequently,
\begin{align*}
    Z_1(x) & = \mathscr S_1(x) + \mathscr S_2(x) + \mathscr S_3(x) + \mathscr S_4(x), \\
    \mathscr S_1(x) & := \sum_{k \in \mathbb Z_0} \left( \frac{M_{k,0}}{\la_{k,0}} - \frac{M_{k,1}}{\la_{k,1}}\right) \sin (2 \la_{k,0} x - Q(x) - \tilde Q(x)), \\
    \mathscr S_2(x) & := \sum_{k \in \mathbb Z_0} \frac{M_{k,1}}{\la_{k,1}} (\sin (2 \la_{k,0} x - Q(x) - \tilde Q(x)) - \sin(2 \la_{k,1} x - Q(x) - \tilde Q(x))), \\
    \mathscr S_3(x) & := \sum_{k \in \mathbb Z_0} \left( \frac{M_{k,0}}{\la_{k,0}} - \frac{M_{k,1}}{\la_{k,1}}\right) \varkappa_{k,0}(x), \\
    \mathscr S_4(x) & := \sum_{k \in \mathbb Z_0} \frac{M_{k,1}}{\la_{k,1}} (\varkappa_{k,0}(x) - \varkappa_{k,1}(x)).
\end{align*}
It follows from \eqref{asymptla1} and \eqref{asymptM1} that
\begin{equation} \label{smMla}
\left| \frac{M_{k,1}}{\la_{k,1}}\right| \le C, \quad
\left\{ \frac{M_{k,0}}{\la_{k,0}} - \frac{M_{k,1}}{\la_{k,1}} \right\} \in l_2.
\end{equation}
Furthermore,
\begin{multline*}
\sin (2 \la_{k,0} x - Q(x) - \tilde Q(x)) - \sin(2 \la_{k,1} x - Q(x) - \tilde Q(x)) \\ = 2 (\la_{k,0} - \la_{k,1})x \cos((\la_{k,0} + \la_{k,1})x - Q(x) - \tilde Q(x)) + O(\xi_k^2).
\end{multline*}
Hence, the series $\mathscr S_1(x)$ and $\mathscr S_2(x)$ converge in $L_2(0, \pi)$. In view of \eqref{sumkappa} and \eqref{smMla}, the series $\mathscr S_3(x)$ and $\mathscr S_4(x)$ converge absolutely and uniformly on $[0, \pi]$. Thus $\eps_1' \in L_2(0, \pi)$, so $\eps_1 \in W_2^1[0, \pi]$.

\textbf{Step 3.} Under our assumptions, we have
$$
\eps_2(x) = \sum_{(k, j) \in J} (-1)^j M_{k,j} \left(\la_{k,j} \tilde S_{k,j}(x) S_{k,j}(x) (\Theta(x))^{-1} - \frac{1}{2 \la_{k,j}} \right).
$$
Lemma~\ref{lem:trans} yields
$$
\la_{k,j}\tilde S_{k,j}(x) S_{k,j}(x) = \frac{1}{2 \la_{k,j}} (\Theta(x) - \cos(2 \la_{k,j}x - Q(x) - \tilde Q(x)) + \varkappa_{k,j}(x)),
$$
where $\{ \varkappa_{k,j}(x) \}$ is some sequence satisfying \eqref{sumkappa}. Consequently,
$$
\eps_2(x) = -\sum_{(k, j) \in J} (-1)^j \frac{M_{k,j}}{\la_{k,j}}(\cos(2 \la_{k,j}x - Q(x) - \tilde Q(x)) - \varkappa_{k,j}(x)).
$$
Analogously to Step~2 of this proof, we show that $\eps_2 \in L_2(0, \pi)$. The proof for $\eps_3$ is similar.
\end{proof}

\begin{lem} \label{lem:findq}
If $\Theta(x) \ne 0$, then $1 + \eps_1^2(x) \ne 0$ and
\begin{align} \label{findTheta}
    \Theta^2(x) = & \, \frac{1}{1 + \eps_1^2(x)}, \quad \Lambda^2(x) = \frac{\eps_1^2(x)}{1 + \eps_1^2(x)}, \quad \Theta(x) \Lambda(x) = \frac{\eps_1(x)}{1 + \eps_1^2(x)},
    \\ \label{findq1}
    q_1(x) = & \, \tilde q_1(x) + \frac{\eps_1'(x)}{1 + \eps_1^2(x)}, \\ \nonumber
    q_0(x) = & \, \tilde q_0(x) + 2 \eps_2'(x) - 2 \tilde q'_1(x) \eps_1(x) - 4 \tilde q_1(x) \eps_1'(x) + 2 (\tilde q_1(x) - q_1(x)) \eps_3(x) \\ \label{findq0} & + b(x) (\eps_2(x) - 2 \tilde q_1(x) \eps_1(x) + \eps_4(x)) + 2 \eps_4'(x) + \frac{b'(x)}{2} + \frac{b^2(x)}{4},     
\end{align}
where $b(x) := 2 (\tilde q_1(x) - q_1(x)) \eps_1(x)$ and the derivatives of $L_2$-functions are understood in the sense of distributions.
\end{lem}

Finally, we arrive at the following algorithm for solving Inverse Problem~\ref{ip:M}.

\begin{alg} \label{alg:1}
Suppose that the data $\{ \la_n, M_n \}_{n \in \mathbb Z_0}$ are given. 
\begin{enumerate}
    \item Choose a model problem $\tilde L = L(\tilde q_0, \tilde q_1)$ such that $\tilde \om_0 = \om_0$ and $\Theta(x) \ne 0$ on $[0, \pi]$.
    \item Construct $\tilde \phi(x)$ and $\tilde H(x)$.
    \item Find $z(x)$ by solving the main equation~\eqref{main}.
    \item Find $\eps_1(x)$ by using \eqref{defvn}-\eqref{defeps1} and then $\Theta(x)$, $\Lambda(x)$ by \eqref{findTheta}.
    \item Calculate the functions $\eps_j(x)$, $j = \overline{2, 4}$, by formulas~\eqref{defeps2}-\eqref{defeps4}.
    \item Find $q_1$ and $q_0$ by \eqref{findq1}-\eqref{findq0}.
\end{enumerate}
\end{alg}

Note that the choice of the square root branch for $\Theta(x)$ and $\Lambda(x)$ is uniquely specified by the continuity of these functions, the condition $\Theta(0) = 1$, and~\eqref{findTheta}. If $\Theta(x) = 0$ for some $x \in [0, \pi]$, one can apply the step-by-step process described in \cite{BY12}. However, in our analysis of the inverse problem solvability and stability in the further sections, the condition $\Theta(x) \ne 0$ is always fulfilled.

Algorithm~\ref{alg:1} implies the following uniqueness theorem for solution of Inverse Problem~\ref{ip:M}.

\begin{thm} \label{thm:uniq}
If $\la_n = \tilde \la_n$ and $M_n = \tilde M_n$, $n \in \mathbb Z_0$, then $q_j = \tilde q_j$ in $W_2^{j-1}(0, \pi)$, $j = 0, 1$. Thus, the spectral data $\{ \la_n, M_n \}_{n \in \mathbb Z_0}$ of the problem $L$ uniquely specify the potentials $q_j$, $j = 0, 1$.
\end{thm}

In the case $q_j \in W_2^j[0, \pi]$, $j = 0, 1$, the series for $\eps_j(x)$ converge in $W_2^{3 - j}[0, \pi]$, $j = \overline{1, 3}$. Moreover, one can use the following simpler formulas instead of~\eqref{defeps2}-\eqref{defeps3}:
\begin{align} \label{simpeps2}
    \eps_2(x) & := \sum_{(k, j) \in J} (-1)^j \la_{k,j} \tilde B_{k,j}(x) v_{k,j}(x), \\ \label{simpeps3}
    \eps_3(x) & := \sum_{(k, j) \in J} (-1)^j \tilde B'_{k,j}(x) v_{k,j}(x).
\end{align}
Usage of either \eqref{defeps2}-\eqref{defeps3} or \eqref{simpeps2}-\eqref{simpeps3} leads to the same $q_0$, $q_1$ in \eqref{findq1}-\eqref{findq0}.

\section{Estimates and approximation} \label{sec:approx}

This section plays an auxiliary role in studying solvability and stability of Inverse Problem~\ref{ip:M}. We impose the assumption of the uniform boundedness of the inverse operator $(I - \tilde H(x))^{-1}$, and obtain auxiliary estimates for the values constructed by Algorithm~\ref{alg:1}. Further, by using the finite spectral data $\{ \la_n, M_n \}_{|n| \le N}$, we construct the infinitely differentiable approximations $q_j^N$ of the potentials $q_j$ in Theorem~\ref{thm:approx}. This theorem plays an auxiliary role in the proofs of global and local solvability, but also can be considered as a separate result.

In this section, we assume that $\tilde L = L(\tilde q_0, \tilde q_1)$, $\tilde q_j \in W_2^j[0, \pi]$, $j = 0, 1$, $\{ \la_n, M_n \}_{n \in \mathbb Z_0}$ are complex numbers (not necessarily being the spectral data of some problem $L$) numbered according to Assumption~($\mathcal O$). Suppose that the numbers $\{ \la_n, M_n \}_{n \in \mathbb Z_0}$ and the spectral data $\{ \tilde \la_n, \tilde M_n \}_{n \in \mathbb Z_0}$ satisfy the following condition
\begin{equation} \label{estOmega}
    \Omega := \sqrt{\sum_{n \in \mathbb Z_0} (n \xi_n)^2} < \iy.
\end{equation}

For $x\in[0, \pi]$, consider the linear bounded operator $\tilde H(x) = \tilde H(\{ \la_n, M_n \}_{n \in \mathbb Z_0}, \tilde L)$ constructed according to the previous section.

\medskip

\textbf{Assumption} $(\mathcal I)$: For each fixed $x \in [0, \pi]$, the operator $(I - \tilde H(x))$ is invertible, and $\| (I - \tilde H(x))^{-1} \|_{\mathfrak B \to \mathfrak B} \le C$ uniformly with respect to $x \in [0, \pi]$.

\medskip

Together with $\tilde H(x)$, consider the operators $\tilde H^N(x) = [\tilde H^N_{n,i; k,j}(x)]_{(n, i), (k, j) \in J}$, $N \ge 1$ defined as
\begin{equation} \label{defHN}
    \tilde H^N_{n,i; k,j}(x) = \left\{\begin{array}{ll}
                                \tilde H_{n,i; k,j}(x), \quad & |k| \le N, \\
                                0, \quad & |k| > N.
                            \end{array}\right.
\end{equation}

Using \eqref{estH}, we derive
\begin{multline*}
    \| \tilde H(x) - \tilde H^N(x) \|_{\mathfrak B \to \mathfrak B} = \sup_{(n,i) \in J} \sum_{\substack{|k| > N \\ j = 0, 1}} |\tilde H_{n,i; k,j}(x)| = \sup_{n \in \mathbb Z_0} \sum_{|k| > N} C \xi_k \left( \frac{1}{|n - k| + 1} + \frac{1}{|k|}\right) \le C \Omega_N,
\end{multline*}
where
$$
\Omega_N := \sqrt{\sum_{|k| > N} (k \xi_k)^2}, \quad \lim_{N \to \iy} \Omega_N = 0.
$$

Therefore, we arrive at the following lemma.

\begin{lem} \label{lem:invHN}
For all sufficiently large $N$, the operators $(I - \tilde H^N(x))$ are invertible in $\mathfrak B$ for each fixed $x \in [0, \pi]$ and 
\begin{gather*}
\| (I - \tilde H^N(x))^{-1} \|_{\mathfrak B \to \mathfrak B} \le C, \\
\| \tilde H(x) - \tilde H^N(x) \|_{\mathfrak B \to \mathfrak B} \le C \Omega_N, \quad \| (I - \tilde H(x))^{-1} - (I - \tilde H^N(x))^{-1} \|_{\mathfrak B \to \mathfrak B} \le C \Omega_N,
\end{gather*}
where $C$ does not depend on $x$ and $N$.
\end{lem}

\begin{lem} \label{lem:invest}
Let $\tilde q_j \in W_2^j[0, \pi]$, $j = 0, 1$, be complex-valued functions, and let $\{ \la_n, M_n \}_{n \in \mathbb Z_0}$ be complex numbers satisfying Assumption $(\mathcal O)$. Suppose that the estimate \eqref{estOmega} and Assumption $(\mathcal I)$ are fulfilled. Then the components $[R_{n,i; k,j}(x)]_{(n,i), (k,j) \in J}$ of the linear bounded operator 
$$
R(x) := (I - \tilde H(x))^{-1} - I
$$
and the components $[z_{n,i}(x)]_{(n,i) \in J}$ of the solution $z(x) = (I - \tilde H(x))^{-1} \tilde \phi(x)$ of the main equation~\eqref{main} belong to $C^1[0, \pi]$ and satisfy the estimates
\begin{gather} \label{estR}
    |R_{n,i; k,j}(x)| \le C \xi_k \left( \frac{1}{|n - k| + 1} + \frac{1}{|k|} + \eta_k\right), \\ \label{defeta}
    \eta_k := \sqrt{\sum_{l \in \mathbb Z_0} \frac{1}{l^2 (|l - k| + 1)^2}}, \quad \{ \eta_k \} \in l_2, \\ \label{estRp}
    |R'_{n,i; k,j}(x)| \le C |n| \xi_k, \\ \label{estz}
    |z_{n,i}^{(\nu)}(x)| \le C |n|^{\nu}, \quad \nu = 0, 1,
\end{gather}
for $(n, i), (k, j) \in J$, $x \in [0, \pi]$.    
\end{lem}

\begin{proof}
\textbf{Step 1.} Let us prove the continuity of $R_{n,i; k,j}(x)$. Clearly, $\tilde H_{n,i; k,j} \in C[0, \pi]$. Fix $\eps > 0$ and choose $N$ such that the conclusion of Lemma~\ref{lem:invHN} holds and 
\begin{equation} \label{smdifRN}
\| R(x) - R^N(x) \|_{\mathfrak B \to \mathfrak B} \le \frac{\eps}{3}, \quad x \in [0, \pi],
\end{equation}
where $R^N(x) = (I - \tilde H^N(x))^{-1} - I$. Note that the inverse $(I - \tilde H^N(x))^{-1}$ can be found by solving the system of finite linear equations
$$
a_{n,i} + \sum_{\substack{|k| \le N \\ j = 0, 1}} \tilde H_{n,i; k,j}(x) a_{k,j} = b_{n,i}, \quad (n, i) \in J.
$$
with respect to $[a_{n,i}]$ by Cramer's rule. Consequently, the components $R^N_{n,i; k,j}(x)$ of the inverse operator are continuous functions. For fixed $n,i,k,j$ and $x_0 \in [0, \pi]$, choose $\de > 0$ such that, for all $x \in [0, \pi] \cap [x_0 - \de, x_0 + \de]$, we have $|R_{n,i; k,j}^N(x) - R_{n,i; k,j}^N(x_0)| \le \frac{\eps}{3}$. This together with \eqref{smdifRN} yield
\begin{multline*}
|R_{n,i; k,j}(x) - R_{n,i; k,j}(x_0)| \le |R_{n,i; k,j}(x) - R^N_{n,i; k,j}(x)| + |R_{n,i; k,j}^N(x) - R_{n,i; k,j}^N(x_0)| \\ + |R_{n,i; k,j}(x_0) - R^N_{n,i; k,j}(x_0)| \le \eps.
\end{multline*}
Thus, $R_{n,i; k,j}(x)$ is continuous at $x_0$. Since $x_0$ is arbitrary, we conclude that $R_{n,i; k,j} \in C[0, \pi]$.

\smallskip

\textbf{Step 2.} Let us estimate $R_{n,i; k,j}(x)$. By definition, 
$$
(I + R(x))(I - \tilde H(x)) = I.
$$
In the element-wise form, this implies
\begin{equation} \label{sumRnk}
R_{n,i; k,j}(x) = \tilde H_{n,i; k,j}(x) + \sum_{(l,s) \in J} R_{n,i;l,s}(x) \tilde H_{l,s; k,j}(x), \quad (n, i), (k, j) \in J.
\end{equation}
Using the estimate $\| R(x) \|_{\mathfrak B \to \mathfrak B} \le C$ and \eqref{estH}, we obtain
\begin{equation} \label{estRrough}
|R_{n,i; k,j}(x)| \le |\tilde H_{n,i; k,j}(x)| + \| R(x) \|_{\mathfrak B \to \mathfrak B} \sup_{(l,s) \in J} |\tilde H_{l,s; k,j}(x)| \le C \xi_k.
\end{equation}
Using \eqref{estRrough}, \eqref{estH}, \eqref{estOmega}, and \eqref{defeta}, we derive
\begin{multline*}
\sum_{(l,s) \in J} |R_{n,i;l,s}(x)| |\tilde H_{l,s; k,j}(x)| \le C \sum_{l \in \mathbb Z_0} \xi_l \xi_k \left( \frac{1}{|l - k| + 1} + \frac{1}{|k|}\right) \\ \le \frac{C \xi_k}{|k|} \sum_{l \in \mathbb Z_0} \xi_l + C \xi_k \sqrt{\sum_{l \in \mathbb Z_0} (l \xi_l)^2} \eta_k \le C \xi_k ( |k|^{-1} + \eta_k).
\end{multline*}
Using this estimate together with \eqref{sumRnk} and \eqref{estH}, we arrive at \eqref{estR}.

\smallskip

\textbf{Step 3.} Let us prove \eqref{estRp}. Since $\tilde q_j \in W_2^j[0, \pi]$, $j = 0, 1$, we have $\tilde H_{n,i; k,j} \in C^1[0, \pi]$ and the estimates \eqref{estHp} hold. Formal differentiation implies
$$
R'(x) = (I + R(x)) \tilde H'(x) (I + R(x)).
$$
Put $G(x) := \tilde H'(x) (I + R(x))$. Then
\begin{align} \label{sumG1}
    G_{n,i; k,j}(x) & = \tilde H'_{n,i; k,j}(x) + \sum_{(l,s) \in J} \tilde H'_{n,i; l,s}(x) R_{l,s; k,j}(x), \\ \label{sumG2}
    R'_{n,i; k,j}(x) & = G_{n,i; k,j}(x) + \sum_{(l, s) \in J} R_{n,i; l,s}(x) G_{l,s; k,j}(x).
\end{align}
Using \eqref{estHp}, \eqref{estOmega}, \eqref{estRrough}, and \eqref{sumG1}, we obtain
\begin{equation} \label{estG}
|G_{n,i; k,j}(x)| \le C |n|\xi_k. 
\end{equation}
Using \eqref{estOmega}, \eqref{estR}, \eqref{sumG2}, and \eqref{estG}, we arrive at \eqref{estRp}. The absolute and uniform convergence of the series in \eqref{sumG1}-\eqref{sumG2} also follows from \eqref{estOmega}, \eqref{estHp}, and \eqref{estR}, so $R_{n,i; k,j} \in C^1[0, \pi]$.

\smallskip

\textbf{Step 4.} Let us estimate $z^{(\nu)}_{n,i}(x)$. Since $z(x) = (I + R(x)) \tilde \phi(x)$, we obtain $\| z(x) \|_{\mathfrak B} \le C$, so \eqref{estz} holds for $\nu = 0$. 
Differentiation implies
\begin{equation} \label{difz}
z'_{n,i}(x) = \sum_{(k,j) \in J} R'_{n,i; k,j}(x) \tilde \phi_{k,j}(x) + \sum_{(k,j) \in J} R_{n,i; k,j}(x) \tilde \phi'_{k,j}(x).
\end{equation}
Using the estimates \eqref{estHp}, \eqref{estOmega}, \eqref{estR}, \eqref{estRp}, and $| \tilde \phi_{n,i}(x) | \le C$, we show that the series in \eqref{difz} converge absolutely and uniformly in $[0, \pi]$, so $z_{n,i} \in C^1[0, \pi]$, and obtain \eqref{estz} for $\nu = 1$.
\end{proof}

Using Lemma~\ref{lem:invest}, it can be shown that the series \eqref{defeps1}, \eqref{simpeps2}, \eqref{simpeps3}, and the series of derivatives for \eqref{defeps1} consist of continuous functions and converge absolutely and uniformly in $[0, \pi]$. Consequently, we obtain the following lemma.

\begin{lem} \label{lem:eps}
Under the conditions of Lemma~\ref{lem:invest}, the function $\eps_1(x)$ defined by \eqref{defeps1} belongs to $C^1[0, \pi]$, the functions $\eps_j(x)$, $j = 2,3,4$, defined by \eqref{defeps4}, \eqref{simpeps2}-\eqref{simpeps3} are continuous on $[0, \pi]$, and
\begin{equation} \label{esteps}
|\eps_1^{(\nu)}(x)| \le C \Omega, \quad \nu = 0, 1, \qquad
|\eps_j(x)| \le C \Omega, \quad j = 2,3,4, \quad x \in [0, \pi].
\end{equation}
\end{lem}

Below we consider two problems $L(q_0, q_1)$ and $L(\tilde q_0, \tilde q_1)$ with the spectral data $\{ \la_n, M_n \}_{n \in \mathbb Z_0}$ and $\{ \tilde \la_n, \tilde M_n \}_{n \in \mathbb Z_0}$, respectively, numbered according to Assumption $(\mathcal O)$. For $N \in \mathbb N$, define the data $\{ \la_n^N, M_n^N \}_{n \in \mathbb Z_0}$ as follows:
\begin{equation} \label{deflaN}
\la_n^N = \begin{cases}
            \la_n, \quad |n| \le N, \\
            \tilde \la_n, \quad |n| > N,
        \end{cases}
\qquad
M_n^N = \begin{cases}
            M_n, \quad |n| \le N, \\
            \tilde M_n, \quad |n| > N.
        \end{cases}
\end{equation}

\begin{thm} \label{thm:approx}
Suppose that $q_j \in W_2^j[0, \pi]$, $j = 0, 1$. Let the functions $\tilde q_j \in C^{\iy}[0, \pi]$, $j = 0, 1$, be such that $\om_k = \tilde \om_k$, $k = \overline{0, 2}$, $\Theta(x) \ne 0$ for all $x \in [0, \pi]$, and Assumption $(\mathcal I)$ is fulfilled. Then, for every sufficiently large $N$, the numbers $\{ \la_n^N, M_n^N \}_{n \in \mathbb Z_0}$ are the spectral data of the problem $L(q_0^N, q_1^N)$ with some functions $q_j^N \in C^{\iy}[0, \pi]$. In addition,
\begin{equation} \label{estqN}
\left| \int_0^x (q_0(t) - q_0^N(t)) \, dt\right| \le C \Omega_N, \quad
|q_1(x) - q_1^N(x)| \le C \Omega_N, \quad x \in [0, \pi],
\end{equation}
where the constant $C$ does not depend on $x$ and $N$.
\end{thm}

\begin{proof}
\textbf{Step 1.} At the first step, we obtain auxiliary estimates. It follows from the condition $\om_k = \tilde \om_k$, $k = \overline{0, 2}$, and the asymptotics \eqref{asymptla2}-\eqref{asymptM2} that \eqref{estOmega} holds. Consider the operator $\tilde H(x) = \tilde H(\{ \la_n, M_n \}_{n \in \mathbb Z_0}, \tilde L)$, $\tilde L = L(\tilde q_0, \tilde q_1)$, and the operators $\tilde H^N(x) = \tilde H(\{ \la_n^N, M_n^N \}_{n \in \mathbb Z_0}, \tilde L)$, $N \ge 1$. Clearly, the functions $\tilde q_0$, $\tilde q_1$ together with $\{ \la_n, M_n \}_{n \in \mathbb Z_0}$ satisfy the conditions of Lemma~\ref{lem:invest}, so the estimates \eqref{estR}, \eqref{estRp}-\eqref{estz} hold. Note that the operators $\tilde H^N(x)$, $N \ge 1$, coincide with the ones defined by \eqref{defHN}. By virtue of Lemma~\ref{lem:invHN}, for all sufficiently large $N$, the functions $\tilde q_0$, $\tilde q_1$ together with the data $\{ \la_n^N, M_n^N \}_{n \in \mathbb Z_0}$ satisfy the conditions of Lemma~\ref{lem:invest}. At Steps 1-2 of this proof, we agree that, if a symbol $\ga$ denotes an object constructed by $\{ \la_n, M_n \}_{n \in \mathbb Z_0}$ and $\tilde L$, the symbol $\ga^N$ with the upper index $N$ will denote the similar object constructed by $\{ \la_n^N, M_n^N \}_{n \in \mathbb Z_0}$ and $\tilde L$. Lemmas~\ref{lem:invHN} and \ref{lem:invest} imply the estimates
\begin{gather} \label{estdifRN}
    \| R(x) - R^N(x) \|_{\mathfrak B \to \mathfrak B} \le C \Omega_N, \\
    \label{estRN}
    |R^N_{n,i; k,j}(x)| \le C \xi_k \left( \frac{1}{|n - k| + 1} + \frac{1}{|k|} + \eta_k \right), \quad |(R^N_{n,i; k,j}(x))'| \le C |n|\xi_k, \\ \label{estzN}
    |(z_{n,i}(x))^{(\nu)}| \le C |n|^{\nu}, \quad \nu = 0, 1,
\end{gather}
for $(n, i), (k, j) \in J$, $x \in [0, \pi]$, where the constant $C$ does not depend on $N$.

In view of \eqref{defchi}, \eqref{defphi}, and \eqref{deflaN}, we have
$\tilde \phi_{n,i}^N(x) = \tilde \phi_{n,i}(x)$ for $|n| \le N$ and $\tilde \phi_{n,i}^N(x) = 0$ for $|n| > N$. Hence
\begin{align} \nonumber
    z_{n,i}(x) & = \tilde \phi_{n,i}(x) + \sum_{(k,j) \in J} R_{n,i; k,j}(x) \tilde \phi_{k,j}(x), \\ \label{sumzN}
    z_{n,i}^N(x) & = \tilde \phi_{n,i}(x) + \sum_{\substack{|k| \le N \\ j = 0, 1}} R_{n,i; k,j}^N(x) \tilde \phi_{k,j}(x)
\end{align}
for $|n| \le N$, $i = 0, 1$. Applying the subtraction and the estimates \eqref{estR}, \eqref{estdifRN}, $|\phi_{k,j}(x)| \le C$, we derive
\begin{multline} \label{difzN}
|z_{n,i}(x) - z_{n,i}^N(x)| \le \sum_{\substack{|k| \le N \\ j = 0, 1}} |R_{n,i; k,j}(x) - R_{n,i; k,j}^N(x)| |\tilde \phi_{k,j}(x)| + \sum_{\substack{|k| > N \\ j = 0, 1}} |R_{n,i; k,j}(x)||\tilde \phi_{k,j}(x)| \\ \le C \Omega_N + C \sum_{|k| > N}\xi_k \le C \Omega_N,
\end{multline}
for $|n| \le N$, $i = 0, 1$, $x \in [0, \pi]$. Following the strategy of the proof of Lemma~\ref{lem:invest} and using the estimates \eqref{estR}, \eqref{estRp}-\eqref{estz}, \eqref{estdifRN}-\eqref{estzN}, we analogously obtain
\begin{gather} \nonumber
|R_{n,i; k,j}(x) - R_{n,i; k,j}^N(x)| \le C \Omega_N \xi_k \left( \frac{1}{|k|} + \eta_k\right), \quad
|R'_{n,i; k,j}(x) - (R_{n,i; k,j}^N(x))'| \le C \Omega_N |n|\xi_k, \\ \label{difzNp}
|z'_{n,i}(x) - (z^N_{n,i}(x))'| \le C \Omega_N |n|,
\end{gather}
where $|n|, |k| \le N$, $i, j = 0, 1$, $x \in [0, \pi]$. 
Using \eqref{defvn}, \eqref{difzN}-\eqref{difzNp}, we derive the estimates
\begin{gather*}
|(v_{n,i}(x) - v_{n,i}^N(x))^{(\nu)}| \le C \Omega_N |n|^{1 - \nu}, \\
|(v_{n,0}(x) - v_{n,1}(x) - v_{n,0}^N(x) + v_{n,1}^N(x))^{(\nu)}| \le C \Omega_N |n|^{1-\nu} \xi_n
\end{gather*}
for $\nu = 0, 1$, $|n| \le N$, $i = 0, 1$, $x \in [0, \pi]$. Taking the latter estimates, formulas \eqref{defeps1}, \eqref{defeps4}, \eqref{simpeps2}, \eqref{simpeps3}, and Lemma~\ref{lem:eps} into account, we conclude that $\eps_1, \eps_1^N \in C^1[0, \pi]$, $\eps_j, \eps_j^N \in C[0, \pi]$, $j = \overline{2, 4}$, and
\begin{equation} \label{difepsN}
    |(\eps_1(x) - \eps_1^N(x))^{(\nu)}| \le C \Omega_N, \quad \nu = 0, 1, \qquad |\eps_j(x) - \eps_j^N(x)| \le C \Omega_N, \quad j = \overline{2, 4}.
\end{equation}

\smallskip

\textbf{Step 2.} Let us construct the functions $q_0^N$, $q_1^N$ and prove the estimates \eqref{estqN}. The assumption $\tilde q_j \in C^{\iy}[0, \pi]$ implies that $\tilde S_{n,i} \in C^{\iy}[0, \pi]$, $(n,i) \in J$. Consequently, $\tilde \phi_{n,i}$ and $\tilde H_{n,i; k,j}$ also belong to $C^{\iy}[0, \pi]$ for all $(n, i), (k, j) \in J$. In view of~\eqref{defHN}, for sufficiently large $N$, the inverse operator $I + R^N(x) = (I - \tilde H^N(x))^{-1}$ can be found by Cramer's rule, so the components $R^N_{n,i; k,j}(x)$ are also infinitely differentiable. Using \eqref{sumzN} and \eqref{defvn} for $v_{n,i}^N$, we conclude that $z_{n,i}^N, v_{n,i}^N \in C^{\iy}[0, \pi]$, $(n, i) \in J$. Obviously, $\eps_j^N(x)$, $j = \overline{1, 4}$, are finite sums of $C^{\iy}$-functions, so $\eps_j \in C^{\iy}[0, \pi]$, $j = \overline{1, 4}$.

By Lemma~\ref{lem:findq}, the condition $\Theta(x) \ne 0$ implies $\eps_1^2(x) + 1 \ne 0$, $x \in [0, \pi]$. It follows from \eqref{difepsN} and $\lim\limits_{N \to \iy} \Omega_N = 0$ that, for sufficiently large $N$, $(\eps_1^N(x))^2 + 1 \ne 0$ and the functions
$$
(\Theta^N(x))^2 = \frac{1}{1 + (\eps_1^N(x))^2}
$$
are infinitely differentiable and uniformly bounded with respect to $x \in [0, \pi]$ and $N$. One can uniquely choose the square root branch to find $\Theta^N \in C^{\iy}[0, \pi]$ satisfying $\Theta^N(0) = 1$.

Construct the functions $q_0^N$, $q_1^N$ by formulas \eqref{findq1}-\eqref{findq0}, replacing $\eps_j$ by $\eps_j^N$ and $\Theta$ by $\Theta^N$. Clearly, $q_j \in C^{\iy}[0, \pi]$, $j = 0, 1$. The estimates \eqref{difepsN} imply \eqref{estqN}.

\smallskip

\textbf{Step 3.} It remains to prove that $\{ \la_n^N, M_n^N \}_{n \in \mathbb Z_0}$ are the spectral data of the problem $L(q_0^N, q_1^N)$. At this step, we assume that all the considered objects are related to the data $\{ \la_n^N, M_n^N \}_{n \in \mathbb Z_0}$ for a sufficiently large fixed $N$, and the index $N$ will be omitted for brevity.

Construct the function
\begin{equation} \label{defPhi}
    \Phi(x, \la) := \tilde \Phi(x, \la) \Theta(x) + \sum_{k,j} (-1)^j \tilde F_{k,j}(x, \la) v_{k,j}(x) \Theta(x),
\end{equation}
where the summation range for $k, j$ is $|k| \le N$, $j = 0, 1$, and 
\begin{equation} \label{defF}
    \tilde F_{n + \nu, i}(x, \la) := \sum_{p = \nu}^{m_{n,i}-1} \frac{1}{(p - \nu)!} M_{n + p, i} \frac{\partial^{p - \nu}}{\partial \mu^{p - \nu}}\tilde E(x, \la, \mu) \Big|_{\mu = \la_{n,i}}, \quad n \in \mathbb S_i, \: \nu = \overline{0, m_{n,i}-1}, \: i = 0, 1,
\end{equation}
\begin{multline} \label{relE}
    \tilde E(x, \la, \mu) := \frac{\tilde \Phi(x, \la) \tilde S'(x, \mu) - \tilde \Phi'(x, \la) \tilde S(x, \mu)}{\la - \mu} \\
    = \frac{1}{\la - \mu} + \int_0^x (\la + \mu - 2 q_1(t)) \tilde \Phi(t, \la) \tilde S(t, \mu) \, dt.  
\end{multline}
Clearly, $\Phi(x, \la)$ is analytic in $\la \ne \la_{n,i}$ for each fixed $x \in [0, \pi]$ and infinitely differentiable with respect to $x$ for each fixed $\la \ne \la_{n,i}$, $(n, i) \in J$.

\begin{lem}
The function $\Phi(x, \la)$ defined by \eqref{defPhi} is the Weyl function of the problem $L(q_0, q_1)$.
\end{lem}

\begin{proof}
By direct calculations, one can prove the following relations
\begin{equation} \label{relPhi}
    \ell(\Phi) + 2 \la q_1(x) \Phi = \la^2 \Phi, \quad \Phi(0, \la) = 1.
\end{equation}
Let us show that $\Phi(\pi, \la) = 0$. For simplicity, suppose that $m_{n,i} = 1$ for all $|n| \le N$, $i = 0, 1$. The general case requires technical modifications. Using \eqref{defPhi} and the relation $\tilde \Phi(\pi, \la) = 0$, we derive
\begin{equation} \label{smPhi}
\Phi(\pi, \la) = - \tilde \Phi'(\pi, \la) \Theta(\pi) \sum_{k,j} (-1)^j M_{k,j} \frac{\tilde S_{k,j}(\pi) v_{k,j}(\pi)}{\la - \la_{k,j}}.
\end{equation}
Due to our notations, the main equation \eqref{main} is equivalent to the system 
$$
v_{n,i}(x) = \tilde S_{n,i}(x) + \sum_{k,j}(-1)^j M_{k,j} \tilde D(x, \la_{n,i}, \la_{k,j}) v_{k,j}(x), \quad (n, i) \in J, \quad x \in [0, \pi].
$$
Recall that $\{ \la_{n,1} \}$ are the eigenvalues of $\tilde L$, so $\tilde S_{n,1}(\pi) = 0$. Hence
\begin{equation} \label{smvn}
v_{n,1}(\pi) = \sum_{k,j} (-1)^j M_{k,j} \tilde D(\pi, \la_{n,1}, \la_{k,j}) v_{k,j}(\pi), \quad |n| \le N.
\end{equation}
Relations \eqref{defal}, \eqref{relalM}, and \eqref{defD} imply
$$
\tilde D(\pi, \la_{n,1}, \la_{n,1}) = \tilde \al_n = -M_{n,1}^{-1}, \qquad \tilde D(\pi, \la_{n,1}, \la_{k,1}) = 0, \quad n \ne k.
$$
Consequently, \eqref{smvn} takes the form
\begin{equation} \label{smvn1}
\sum_{|k| \le N} \tilde P_{n,1; k,0}(\pi) v_{k,0}(\pi) = 0, \quad |n| \le N.
\end{equation}

Define the $(4N \times 4N)$-matrix $\tilde H_{4N \times 4N}(x) := [\tilde H_{n,i; k,j}(x)]$, $|n|, |k| \le N$, $i, j = 0, 1$. Denote by $I_{4N \times 4N}$ the $(4N \times 4N)$ unit matrix. It follows from the invertibility of the operator $(I - \tilde H(\pi))$ that 
\begin{equation} \label{detH}
\det(I_{4N \times 4N} - \tilde H_{4N \times 4N}(\pi)) \ne 0. 
\end{equation}
By using the definitions \eqref{defA}-\eqref{defP} and \eqref{defH}, one can show that \eqref{detH} implies $\det(\tilde P_{2N \times 2N}(\pi)) \ne 0$, where $\tilde P_{2N \times 2N}(\pi) := [\tilde P_{n,1; k,0}(\pi)]$, $|n|, |k| \le N$. Hence, the system~\eqref{smvn1} has the only zero solution $v_{k,0}(\pi) = 0$, $|k| \le N$.

Since $v_{k,0}(\pi) = 0$ and $\tilde S_{k,1}(\pi) = 0$ in \eqref{smPhi}, we obtain $\Phi(\pi, \la) \equiv 0$. Together with \eqref{relPhi}, this yields the claim of the lemma.
\end{proof}

Proceed with the proof of Theorem~\ref{thm:approx}. In view of~\eqref{defPhi}, the Weyl function has the form
$$
M(\la) = \Phi'(0, \la) = \tilde M(\la) + \sum_{k,j} (-1)^j \tilde F_{k,j}(0, \la).
$$
Taking \eqref{defF} and \eqref{relE} into account, we obtain
$$
M(\la) = \tilde M(\la) + \sum_{j = 0, 1} (-1)^j \sum_{n \in \mathbb S_j, \, |n| \le N} \sum_{\nu = 0}^{m_{n,j}-1} \frac{M_{n + \nu,j}}{(\la - \la_{n,j})^{\nu + 1}}
$$
Clearly, the function $M(\la)$ is meromorphic with the poles $\{ \la_n^N \}_{n \in \mathbb Z_0}$ and the corresponding residues $\{ M_n^N \}_{n \in \mathbb Z_0}$. Thus, $\{ \la_n^N, M_n^N \}_{n \in \mathbb Z_0}$ are the spectral data of the constructed problem $L(q_0^N, q_1^N)$, so the proof of Theorem~\ref{thm:approx} is finished.
\end{proof}

\begin{remark} \label{rem:approx}
Note that, for any fixed functions $q_j \in W_2^j[0, \pi]$, $j = 0, 1$, there exist $\tilde q_j \in C^{\iy}[0, \pi]$, $j = 0, 1$, satisfying the conditions of Theorem~\ref{thm:approx}. Indeed, for every $\de > 0$, one can find polynomials $\tilde q_0$, $\tilde q_1$ such that 
\begin{equation} \label{smmodel}
\om_k = \tilde \om_k, \quad k = \overline{0, 2}, \qquad \| q_j - \tilde q_j \|_{W_2^j[0, \pi]} \le \de, \quad j = 0, 1.
\end{equation}
In particular, \eqref{smmodel} implies $\Theta(x) \ne 0$, $x \in [0, \pi]$, for sufficiently small $\de > 0$. If the spectrum of the problem $L(q_0, q_1)$ is simple, one can easily prove the stability of the direct problem $L(q_0, q_1) \mapsto \{ \la_n, M_n \}_{n \in \mathbb Z_0}$. Namely, the conditions \eqref{smmodel} imply the inequality $\Omega \le C \de$ for every $\de \in (0, \de_0]$ with some $\de_0 > 0$. Consequently, $\| \tilde H(x) \|_{\mathfrak B \to \mathfrak B} \le C \de$, so for sufficiently small $\de$ Assumption $(\mathcal I)$ is fulfilled. The case of multiple eigenvalues can be treated similarly, by using the approach of Section~\ref{sec:mult}.
\end{remark}

\begin{remark}
In view of~\eqref{relalM}, Theorems~\ref{thm:uniq} and~\ref{thm:approx} are valid for the spectral data $\{ \la_n, M_n \}_{n \in \mathbb Z_0}$ being replaced by $\{ \la_n, \al_n \}_{n \in \mathbb Z_0}$.
\end{remark}

\section{Solvability and stability} \label{sec:sol}

The goal of this section is to prove Theorem~\ref{thm:solve} on the global solvability of Inverse Problem~\ref{ip:M}. The proof is based on the constructive solution from Section~\ref{sec:alg}, auxiliary estimates and the approximation by infinitely differentiable potentials obtained in Section~\ref{sec:approx}. Theorem~\ref{thm:solve} implies Corollary~\ref{cor:locsimp} on the local solvability and stability without change of eigenvalue multiplicities. The latter result will be improved in Section~\ref{sec:mult}.

Define the class $C^{-1}[0, \pi]$ of functions $f = g'$, where $g \in C[0, \pi]$ and the derivative is understood in the sense of distributions. Put $C^0[0, \pi] := C[0, \pi]$.

\begin{thm} \label{thm:solve}
Let $\tilde q_j \in W_2^j[0, \pi]$, $j = 0, 1$, be complex-valued functions, and let $\{ \la_n, M_n \}_{n \in \mathbb Z_0}$ be complex numbers satisfying Assumption $(\mathcal O)$. Suppose that the estimate \eqref{estOmega} and Assumption~$(\mathcal I)$ are fulfilled. Then, by Steps~2-4 of Algorithm~\ref{alg:1}, one can construct the function $\eps_1 \in C^1[0, \pi]$. If we additionally assume that $1 + \eps_1^2(x) \ne 0$ for all $x \in [0, \pi]$, then $\{ \la_n, M_n \}_{n \in \mathbb Z_0}$ are the spectral data of the problem $L(q_0, q_1)$, $q_j \in C^{j-1}[0, \pi]$, $j= 0, 1$. The functions $q_0$, $q_1$ can be constructed by formulas~\eqref{findq1}-\eqref{findq0}, where $\eps_k$, $k = \overline{2,4}$, are defined by \eqref{simpeps2}-\eqref{simpeps3} and \eqref{defeps4}. Moreover, the following estimates hold:
\begin{equation} \label{estq}
\left| \int_0^x (q_0(t) - \tilde q_0(t)) \, dt\right| \le C \Omega, \quad
|q_1(x) - \tilde q_1(x)| \le C \Omega, \quad x \in [0, \pi].
\end{equation}
\end{thm}

\begin{proof}
The possibility to construct the functions $\eps_k(x)$, $k = \overline{1, 4}$, of appropriate smoothness follows from Lemmas~\ref{lem:invest}-\ref{lem:eps}. If $1 + \eps_1^2(x) \ne 0$, $x \in [0, \pi]$, the functions $q_0$, $q_1$ of appropriate classes can obviously be constructed by formulas~\eqref{findq1}-\eqref{findq0}. The estimates~\eqref{estq} follow from~\eqref{esteps}.
It remains to show that the spectral data of the problem $L(q_0, q_1)$ coincide with the initially known numbers $\{ \la_n, M_n \}_{n \in \mathbb Z_0}$. For this purpose, we choose the functions $\tilde{\tilde q}_j \in C^{\iy}[0, \pi]$, $j = 0, 1$, such that $\tilde \om_k = \tilde {\tilde \om}_k$, $k = \overline{0, 2}$, $\tilde{\tilde \Theta}(x) := \cos(\tilde Q(x) - \tilde{\tilde Q}(x)) \ne 0$, $x \in [0, \pi]$, and Assumption~$(\mathcal I)$ holds for the operator $\tilde{\tilde H}(x) := \tilde H(\{ \tilde \la_n, \tilde M_n \}_{n \in \mathbb Z_0}, L(\tilde{\tilde q}_0, \tilde{\tilde q}_1))$. In view of Remark~\ref{rem:approx}, such functions $\tilde{\tilde q}_j$, $j = 0, 1$, exist. Define
\begin{gather} \label{deflaN2}
\la_n^N = \begin{cases}
            \la_n, \quad |n| \le N, \\
            \tilde {\tilde \la}_n, \quad |n| > N,
        \end{cases}
\qquad
M_n^N = \begin{cases}
            M_n, \quad |n| \le N, \\
            \tilde {\tilde M}_n, \quad |n| > N.
        \end{cases}, 
\\ \nonumber
\tilde \la_n^N = \begin{cases}
            \tilde \la_n, \quad |n| \le N, \\
            \tilde {\tilde \la}_n, \quad |n| > N,
        \end{cases}
\qquad
\tilde M_n^N = \begin{cases}
            \tilde M_n, \quad |n| \le N, \\
            \tilde {\tilde M}_n, \quad |n| > N.
        \end{cases}, 
\end{gather}

By virtue of Theorem~\ref{thm:approx}, for every sufficiently large $N \in \mathbb N$, the numbers $\{ \tilde \la_n^N, \tilde M_n^N \}_{n \in \mathbb Z_0}$ are the spectral data of the problem $\tilde L^N := L(\tilde q_0^N, \tilde q_1^N)$ with $\tilde q_j^N \in C^{\iy}[0, \pi]$, $j = 0, 1$. Consider the operators $\tilde H(x) := \tilde H(\{ \la_n, M_n \}, \tilde L)$ and $\tilde H^N(x) := \tilde H(\{ \la_n^N, M_n^N \}, \tilde L^N)$. According to Theorem~\ref{thm:approx}, the estimate~\eqref{estH}, and the introduced notations,
\begin{gather*}
|\tilde H^N_{n,i; k,j}(x) - \tilde H_{n,i; k,j}(x)| \le C \Omega_N \xi_k \left( \frac{1}{|n - k| + 1} + \frac{1}{|k|}\right), \quad (n, i), (k, j) \in J, \quad x \in (0, \pi), \\
\tilde H^N_{n,i; k,j}(x) = 0, \quad |k| > N.
\end{gather*}
Hence
$$
\| \tilde H^N(x) - \tilde H(x) \|_{\mathfrak B \to \mathfrak B} \le C \Omega_N.
$$
Since Assumption $(\mathcal I)$ holds for $\tilde H(x)$, it also holds for $\tilde H^N(x)$ with sufficiently large $N$. Therefore, the conditions of Lemmas~\ref{lem:invest}-\ref{lem:eps} are fulfilled for $\{ \la_n^N, M_n^N \}_{n \in \mathbb Z_0}$ together with the problem $\tilde L^N$. Applying these lemmas to $\{ \la_n^N, M_n^N \}_{n \in \mathbb Z_0}$ and $\tilde L^N$, one can construct the infinitely differentiable functions $\eps_j^N(x)$, $j = \overline{1, 4}$, which satisfy the estimates~\eqref{difepsN}. If $1 + \eps_1^2(x) \ne 0$, $x \in [0, \pi]$, we have $1 + (\eps_1^N(x))^2 \ne 0$ for every sufficiently large $N$ and all $x \in [0, \pi]$. Therefore, one can construct by~\eqref{findq1}-\eqref{findq0} the functions $q_j^N \in C^{\iy}[0, \pi]$, $j = 0, 1$, which satisfy~\eqref{estqN}. Since the series for $\eps_j^N(x)$, $j = \overline{1, 4}$, are finite, it can be shown that $\{ \la_n^N, M_n^N \}_{n \in \mathbb Z_0}$ are the spectral data of $L(q_0^N, q_1^N)$ similarly to Step~3 of the proof of Theorem~\ref{thm:approx}.

Using Lemma~\ref{lem:trans} and relations~\eqref{defal}-\eqref{relalM}, we obtain the following auxiliary lemma.

\begin{lem} \label{lem:lim}
Let $\{ \la_n, M_n \}_{n \in \mathbb Z_0}$ and $\{ \la_n^N, M_n^N \}_{n \in \mathbb Z_0}$, $N \ge 1$, be the spectral data of the problems $L(q_0, q_1)$ and $L(q_0^N, q_1^N)$, $N \ge 1$, respectively, where $q_j, q_j^N \in W_2^{j-1}(0, \pi)$, $j = 0, 1$, and
\begin{equation} \label{limq}
\lim_{N \to \iy} \| q_j - q_j^N \|_{W_2^{j-1}(0, \pi)} = 0, \quad j = 0, 1.
\end{equation}
Then, for each fixed $n \in \mathbb Z_0$, we have $\lim\limits_{N \to \iy} \la_n^N = \la_n$. In addition, if $n \in \mathbb S \cap \mathbb S^N$ and $m_n = m_n^N$, then $\lim\limits_{N \to \iy} M_n^N = M_n$.
\end{lem}

Clearly, the functions $q_j$ and $q_j^N$, $j = 0, 1$, constructed in the proof of Theorem~\ref{thm:solve} satisfy \eqref{limq} by virtue of~\eqref{difepsN}. Thus, the spectral data $\{ \la_n^N, M_n^N \}$ of the problem $L(q_0^N, q_1^N)$ converge to the spectral data of the problem $L(q_0, q_1)$ in the sense of Lemma~\ref{lem:lim}. Taking~\eqref{deflaN2} into account, we conclude that the spectral data of $L(q_0, q_1)$ coincide with $\{ \la_n, M_n \}_{n \in \mathbb Z_0}$, so Theorem~\ref{thm:solve} is proved.
\end{proof}

The following corollary of Theorem~\ref{thm:solve} provides local solvability and stability of Inverse Problem~\ref{ip:M}.

\begin{cor} \label{cor:locsimp}
Let $\tilde q_j \in W_2^j[0, \pi]$, $j = 0, 1$, be complex-valued functions, and let $\{ \tilde \la_n, \tilde M_n \}_{n \in \mathbb Z_0}$ be the spectral data of the problem $\tilde L = L(\tilde q_0, \tilde q_1)$. Then there exists $\de_0 > 0$ such that, for any complex numbers $\{ \la_n, M_n \}_{n \in \mathbb Z_0}$ satisfying Assumption $(\mathcal O)$ and the estimate $\Omega \le \de_0$, there exist complex-valued functions $q_j \in C^{j-1}[0, \pi]$, $j = 0, 1$, such that $\{ \la_n, M_n \}_{n \in \mathbb Z_0}$ are the spectral data of the problem $L(q_0, q_1)$. In addition, the estimate~\eqref{estq} is valid.
\end{cor}

\begin{proof}
It follows from \eqref{estH} and \eqref{estOmega} that $\| \tilde H(x) \|_{\mathfrak B \to \mathfrak B} \le C \Omega$, $x \in [0, \pi]$. If $\Omega$ is sufficiently small, Assumption $(\mathcal I)$ is fulfilled. Lemma~\ref{lem:eps} implies $\| \eps_1 \|_{C[0, \pi]} \le C \Omega$, so $1 + \eps_1^2(x) \ne 0$ for sufficiently small $\Omega$. Thus, Theorem~\ref{thm:solve} yields the claim.
\end{proof}

In view of the definitions \eqref{defxi} and \eqref{estOmega}, the multiplicities in the sequences $\{ \la_n \}_{n \in \mathbb Z_0}$ and $\{ \tilde \la_n \}_{n \in \mathbb Z_0}$ coincide for sufficiently small $\Omega$. In the next section, Corollary~\ref{cor:locsimp} will be generalized to the case of changing eigenvalue multiplicities.

\section{Multiple eigenvalue splitting} \label{sec:mult}

In this section, we obtain the local solvability and stability of Inverse problem~\ref{ip:M} in the general case, taking the possible splitting of multiple eigenvalues into account.

Consider a fixed problem $\tilde L = L(\tilde q_0, \tilde q_1)$ with $\tilde q_j \in W_2^j[0, \pi]$, $j = 0, 1$. Fix an index $n_* \in \mathbb N \cup \{ 0 \}$ and a contour $\ga := \{ \la \in \mathbb C \colon |\la| = r \}$, $r > 0$, such that $\tilde m_n = 1$ for $|n| > n_*$, $\tilde \la_n \in \mbox{int} \, \ga$ for all $|n| \le n_*$ and $\tilde \la_n \not\in \overline{\mbox{int} \, \ga}$ for all $|n| > n_*$. Along with $\tilde L$, consider some complex numbers $\{ \la_n, M_n \}_{n \in \mathbb Z_0}$ (not necessarily being the spectral data of some problem $L$). Suppose that Assumption $(\mathcal O)$ holds for the both collections $\{ \la_n, M_n \}_{n \in \mathbb Z_0}$ and $\{ \tilde \la_n, \tilde M_n \}_{n \in \mathbb Z_0}$. Set
\begin{gather*}
    \mathbb S_* := \{n \in \mathbb S \colon |n| \le n_* \}, \quad
    \tilde {\mathbb S}_* := \{ n \in \tilde {\mathbb S} \colon  |n| \le n_* \}, \\
    M_*(\la) := \sum_{n \in \mathbb S_*} \sum_{\nu = 0}^{m_n - 1} \frac{M_{n + \nu}}{(\la - \la_n)^{\nu + 1}}, \quad
    \tilde M_*(\la) := \sum_{n \in \tilde{\mathbb S}_*} \sum_{\nu = 0}^{\tilde m_n - 1} \frac{\tilde M_{n + \nu}}{(\la - \tilde \la_n)^{\nu + 1}}, \quad \hat M_* := M_* - \tilde M_*.
\end{gather*}

\begin{thm} \label{thm:locmult}
Let $\tilde q_j \in W_2^j[0, \pi]$, $j=0,1$. Then there exists $\de_0 > 0$ such that, for any complex numbers $\{ \la_n, M_n \}_{n \in \mathbb Z_0}$ satisfying Assumption $(\mathcal O)$ and the estimate
\begin{equation} \label{estde}
\de := \max \left\{ \max_{\la \in \ga} |\hat M_*(\la)|, \sqrt{\sum_{|n| > n_*} (n \xi_n)^2} \right\} \le \de_0,
\end{equation}
there exist the functions $q_j \in C^{j-1}[0, \pi]$, $j = 0, 1$, such that $\{ \la_n, M_n \}_{n \in \mathbb Z_0}$ are the spectral data of the problem $L(q_0, q_1)$. In addition,
\begin{equation} \label{estqde}
\left| \int_0^x (q_0(t) - \tilde q_0(t)) \, dt\right| \le C \de, \quad
|q_1(x) - \tilde q_1(x)| \le C \de, \quad x \in [0, \pi].
\end{equation}
\end{thm}

We emphasize that the multiplicities in the sequences $\{ \la_n \}_{n \in \mathbb Z_0}$ and $\{ \tilde \la_n \}_{n \in \mathbb Z_0}$ may differ. However, for sufficiently small $\de > 0$, we have $\mathbb S \subseteq \tilde {\mathbb S}$. Roughly speaking, multiple eigenvalues can split into smaller groups but cannot join into new groups.

\begin{proof}
\textbf{Step 1.} Consider the following special case. Let $\tilde q_j \in C^{j-1}[0, \pi]$, $j = 0, 1$, be fixed, and let $\{ \la_n, M_n \}_{n \in \mathbb Z_0}$ be arbitrary numbers satisfying Assumption $(\mathcal O)$ and $\la_n = \tilde \la_n$, $M_n = \tilde M_n$ for all $|n| > n_*$.
If $\de$ defined by \eqref{estde} is sufficiently small, then, by virtue of Lemma~\ref{lem:rat}, $\la_n \in \mbox{int}\, \ga$ for $|n| \le n_*$.

Denote by $C(\ga)$ the Banach space of functions continuous on $\ga$ with the norm $\| f \|_{C(\ga)} := \max\limits_{\la \in \ga} |f(\la)|$. For each fixed $x \in [0, \pi]$, define the linear bounded operator $\tilde H_{\ga}(x) \colon C(\ga) \to C(\ga)$ acting as follows:
$$
(\tilde H_{\ga}(x) f)(\la) = \frac{1}{2 \pi i} \oint\limits_{\ga} \tilde D(x, \la, \mu) \hat M_*(\mu) f(\mu) \, d\mu, \quad f \in C(\ga).
$$

It follows from \eqref{defD} and \eqref{estde} that $\| \tilde H_{\ga}(x) \|_{C(\ga) \to C(\ga)} \le C \de$, $x \in [0, \pi]$. Therefore, there exist $\de_0 > 0$ such that, for any collection $\{ \la_n, M_n \}_{n \in \mathbb Z_0}$ satisfying \eqref{estde} and each $x \in [0, \pi]$, the operator $(I - \tilde H_{\ga}(x))$ is invertible. Moreover, $\| (I - \tilde H_{\ga}(x))^{-1} \|_{C(\ga) \to C(\ga)} \le C$ uniformly with respect to $\de \le \de_0$ and $x \in [0, \pi]$. Hence, for each fixed $x \in [0, \pi]$, the equation
\begin{equation} \label{mainv}
v(x, \la) = \tilde S(x, \la) + \frac{1}{2 \pi i} \oint\limits_{\ga} \tilde D(x, \la, \mu) \hat M_*(\mu) v(x, \mu) \, d\mu
\end{equation}
has the unique solution $v(x, .) \in C(\ga)$, $|v(x, \mu)| \le C$ for $x \in [0, \pi]$ and $\mu \in \ga$.

Define the functions
\begin{align} \label{defE12}
\mathcal E_j(x) & = \frac{1}{2 \pi i} \oint\limits_{\ga} \mu^{j - 1} \hat M_*(\mu) \tilde S(x, \mu) v(x, \mu) \, d\mu, \quad j = 1, 2, \\ \label{defE3}
\mathcal E_3(x) & = \frac{1}{2 \pi i} \oint\limits_{\ga} \hat M_*(\mu) \tilde S'(x, \mu) v(x, \mu) \, d\mu.
\end{align}
It can be easily shown that 
\begin{align} \label{estE12}
\mathcal E_j \in C^1[0, \pi], \quad & \| \mathcal E_j \|_{C^1[0, \pi]} \le C \de, \quad j = 1, 2, \\ \label{estE3}
\mathcal E_3 \in C[0, \pi], \quad & \| \mathcal E_3 \|_{C[0, \pi]} \le C \de.
\end{align}

Note that relation \eqref{mainv} provides the analytical continuation of the function $v(x, \la)$ to the whole complex plane. Calculating the integral in \eqref{mainv} by the residue theorem, we obtain
\begin{equation} \label{discv}
v_{n,i}(x) = \tilde S_{n,i}(x) + \sum_{k,j} (-1)^j \tilde P_{n,i; k,j}(x) v_{k,j}(x), \quad (n, i) \in J, \quad x \in [0, \pi],
\end{equation}
where 
$$
v_{k + \nu,i}(x) := \frac{1}{\nu!} \frac{\partial^{\nu} v(x, \la)}{\partial \la^{\nu}}\Big|_{\la = \la_{k,i}}, \quad k \in \mathbb S_i, \quad \nu = \overline{0, m_{k,i}-1}, \quad i = 0, 1.
$$
The summation in \eqref{discv} can be taken either over $|k| \le n_*$, $j = 0, 1$, or over $(k, j) \in J$, because in our special case $\tilde P_{n,i; k,0}(x) = \tilde P_{n,i; k,1}(x)$, $v_{k,0}(x) = v_{k,1}(x)$ for $|k| > n_*$. Comparing \eqref{discv} with \eqref{rel-cont}, we conclude that the sequence $[v_{n,i}(x)]_{(n, i) \in J}$ coincide with the one defined via \eqref{defvn} by the solution $z(x)$ of the main equation~\eqref{main}. Calculating the integrals in \eqref{defE12}-\eqref{defE3} by the residue theorem, we conclude that
$$
\mathcal E_1(x) = \eps_1(x), \quad \mathcal E_2(x) = \eps_2(x) + \eps_4(x), \quad \mathcal E_3(x) = \eps_3(x),
$$
where $\eps_j(x)$, $j = \overline{1, 4}$, are defined by \eqref{defeps1}, \eqref{defeps4}, \eqref{simpeps2}-\eqref{simpeps3}. Hence, the estimate \eqref{estE12} implies that $1 + \eps_1^2(x) \ne 0$ for all sufficiently small $\de$ and $x \in [0, \pi]$. Thus, by using $\mathcal E_k(x)$, $k = \overline{1, 3}$, one can construct the functions $q_j \in C^{j-1}[0, \pi]$, $j = 0, 1$, by \eqref{findq1}-\eqref{findq0}. Since the sums for $\mathcal E_j(x)$, $j = \overline{1, 3}$, are finite, one can easily show that $\{ \la_n, M_n \}_{n \in \mathbb Z_0}$ are the spectral data of $L(q_0, q_1)$. The estimates \eqref{estE12}-\eqref{estE3} imply \eqref{estqde}. 

\smallskip

\textbf{Step 2.} Consider the general case. Suppose that $q_j \in W_2^j[0, \pi]$, $j = 0, 1$, are fixed. Let $\{ \la_n, M_n \}_{n \in \mathbb Z_0}$ be arbitrary complex numbers satisfying Assumption $(\mathcal O)$ and $\de < \iy$. By virtue of Lemma~\ref{lem:rat}, if $\de$ is sufficiently small, then $\la_n \in \mbox{int}\, \ga$ for $|n| \le n_*$ and $\la_n \not\in \overline{\mbox{int} \, \ga}$ for $|n| > n_*$.

Define the numbers
$$
\check \la_n := \begin{cases} 
                    \tilde \la_n, \quad |n| \le n_*, \\
                    \la_n, \quad |n| > n_*, 
                \end{cases}
\qquad                
\check M_n := \begin{cases} 
                    \tilde M_n, \quad |n| \le n_*, \\
                    M_n, \quad |n| > n_*.
                \end{cases}
$$
By virtue of Corollary~\ref{cor:locsimp}, there exists $\de_0 > 0$ such that, for any collection $\{ \la_n, M_n \}_{n \in \mathbb Z_0}$ satisfying Assumption~$(\mathcal O)$ and the estimate \eqref{estde}, there exist complex-valued functions $\check q_j \in C^{j-1}[0, \pi]$, $j = 0, 1$, such that $\{ \check \la_n, \check M_n \}_{n \in \mathbb Z_0}$ are the spectral data of $L(\check q_0, \check q_1)$ and
\begin{equation} \label{estqc}
\left| \int_0^x (\check q_0(t) - \tilde q_0(t)) \, dt\right| \le C \de, \quad
|\check q_1(x) - \tilde q_1(x)| \le C \de
\end{equation}
uniformly with respect to $x \in [0, \pi]$ and $\de \le \de_0$. Then, by using $\check q_j$ instead of $\tilde q_j$ at Step~1 of this proof, one can construct the problem $L(q_0, q_1)$ with $q_j \in C^{j-1}[0, \pi]$, $j = 0, 1$, having the spectral data $\{ \la_n, M_n \}_{n \in \mathbb Z_0}$. It can be shown that the final estimates \eqref{estqde} are uniform with respect to $\check q_0$, $\check q_1$ if the estimates \eqref{estde} and \eqref{estqc} are fulfilled for sufficiently small $\de_0 > 0$.  
\end{proof}

Note that the conditions of Theorem~\ref{thm:locmult} are formulated in terms of the rational function $\hat M_*(\la)$ constructed by a finite number of the spectral data.

\begin{remark}
The function $\hat M_*(\la)$ in Theorem~\ref{thm:locmult} can be replaced by $\hat M(\la) = M(\la) - \tilde M(\la)$. Indeed, this replacement does not change the contour integrals in the proof of Theorem~\ref{thm:locmult} for sufficiently small $\de$.
\end{remark}

It can be also useful to formulate the local solvability and stability conditions in terms of the discrete data. The following corollary provides such conditions for the case when every multiple eigenvalue $\tilde \la_n$ splits into simple eigenvalues $\la_n$.

\begin{cor} \label{cor:disc}
Let $\tilde q_j \in W_2^j[0, \pi]$, $j = 0, 1$. Then there exists $\de_0 > 0$ such that, for every $\de \in (0, \de_0]$ and any complex numbers $\{ \la_n, M_n \}_{n \in \mathbb Z_0}$ satisfying Assumption $(\mathcal O)$ and the conditions
\begin{gather*}
    \sqrt{\sum_{|n| > n_*} (n \xi_n)^2} \le \de, \\
    \la_n \ne \la_k, \quad n \ne k, \quad n, k \in \mathbb Z_0, \\
    \left| \sum_{\nu = 0}^{\tilde m_k - 1} (\la_{k + \nu}- \tilde \la_k)^s M_{k + \nu} - \tilde M_{k + s}\right| \le \de, \quad s = \overline{0, \tilde m_k-1}, \quad k \in \tilde {\mathbb S}_* \\
    \left| \sum_{\nu = 0}^{\tilde m_k-1} (\la_{k + \nu} - \tilde \la_k)^s M_{k + \nu}\right| \le \de, \quad s = \overline{\tilde m_k, 2(\tilde m_k - 1)}, \quad k \in \tilde {\mathbb S}_*, \\
    |\la_{k + \nu} - \tilde \la_k| \le \de^{1/\tilde m_k}, \quad |M_{k + \nu}| \le \de^{(1-\tilde m_k)/\tilde m_k}, \quad \nu = \overline{0, \tilde m_k-1}, \quad k \in \tilde {\mathbb S}_*,
\end{gather*}
there exist functions $q_j \in C^{j-1}[0, \pi]$ such that $\{ \la_n, M_n \}_{n \in \mathbb Z_0}$ are the spectral data of $L(q_0, q_1)$. In addition, the estimate \eqref{estqde} is valid.
\end{cor}

Corollary~\ref{cor:disc} is proved analogously to Theorem~2.3 in \cite{Bond20}.

\section{Numerical examples} \label{sec:ex}

In this section, we construct an example of a pencil having a double eigenvalue. Then, we approximate this pencil by pencils with simple eigenvalues.

Put $\tilde \la_1 = \tilde \la_{-1} = 0.5$, $\tilde M_{-1} = -\frac{1}{\pi}$, $\tilde M_1 = -\frac{\mathrm{i}}{2\pi}$. This means
$$
\tilde M(\la) \sim \frac{\tilde M_{-1}}{(\la - \la_1)} + \frac{\tilde M_1}{(\la - \la_1)^2}
$$
in a neighbourhood of $\la_1$. For $|n| > 1$, we suppose that the spectral data coincide with the spectral data of the problem $L(0, 0)$, namely, $\tilde \la_n = n$, $\tilde M_n = -\frac{n}{\pi}$. Denote
\begin{gather*}
a := \frac{\tilde M_1}{2}, \quad c := \frac{\tilde M_{-1}}{a}, \quad \lambda_1 := \tilde \lambda_1 + \sqrt{\delta}, \quad \lambda_{-1} := \tilde \lambda_1 - \sqrt{\delta} + c \delta, \\
M_1 := \frac{a}{\sqrt{\delta}} + \tilde M_{-1}, \quad M_{-1} := -\frac{a}{\sqrt{\delta}},
 \qquad \lambda_n := \tilde \lambda_n,
\quad M_n := \tilde M_n, \quad |n| > 1, \quad \de > 0.
\end{gather*}

Observe that, for sufficiently small $\de > 0$, the defined data fulfills the conditions of Corollary~\ref{cor:disc}. An interesting feature of this example is that the eigenvalues $\la_{\pm 1}$ are $\sqrt{\de}$-close to $\tilde \la_{\pm 1}$ and the absolute values of the residues $M_{\pm 1}$ tend to infinity as $\de \to 0$, but the corresponding potentials $q_0$, $q_1$ are $C\de$-close to $\tilde q_0$, $\tilde q_1$ in the sense of the estimate~\eqref{estqde}. This feature is confirmed by numerical computations. For $\de = 0.02$, the plots of the potentials $q_1(x)$, $\tilde q_1(x)$ and $q_0(x)$, $\tilde q_0(x)$ are presented in Figures~\ref{fig:1} and~\ref{fig:2}, respectively.

\begin{figure}[h!]
\begin{center}
\includegraphics[scale = 0.15]{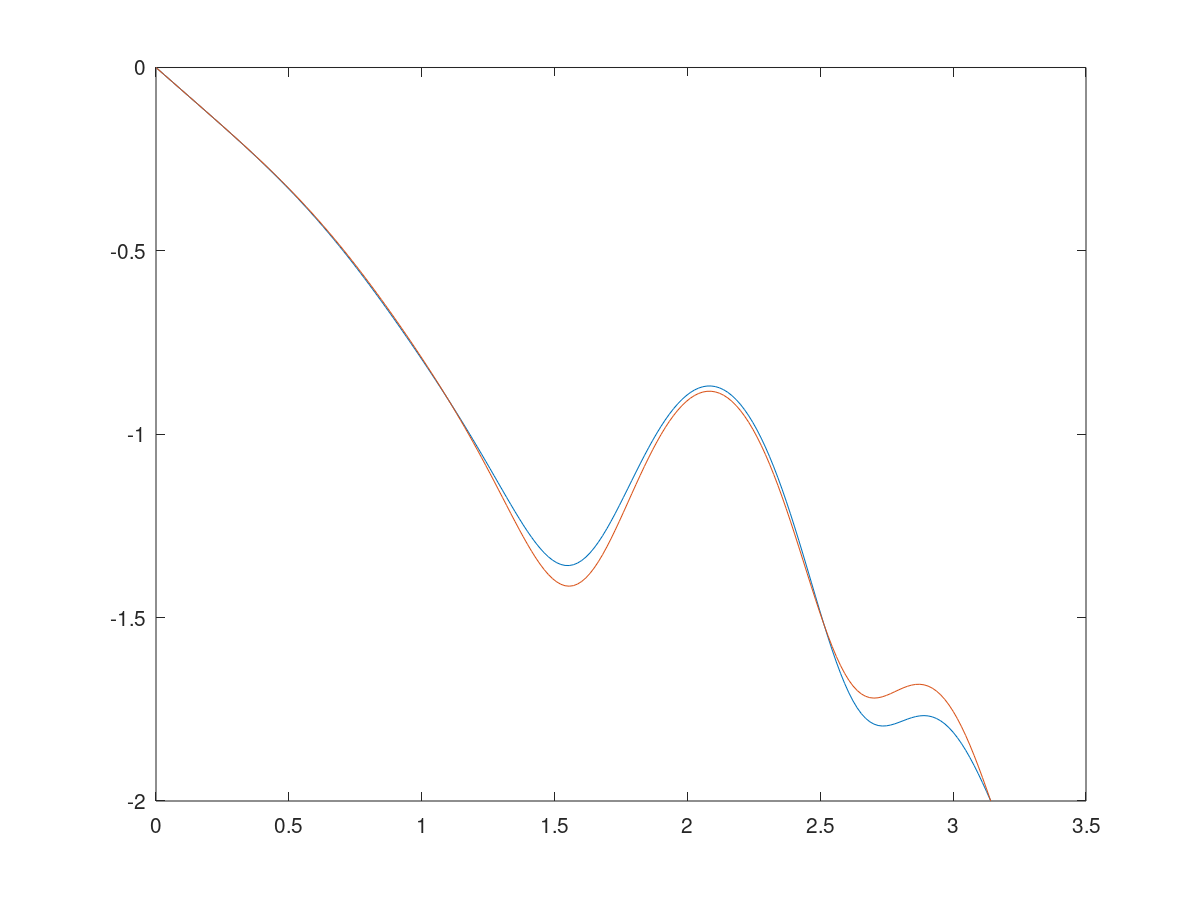}
\includegraphics[scale = 0.15]{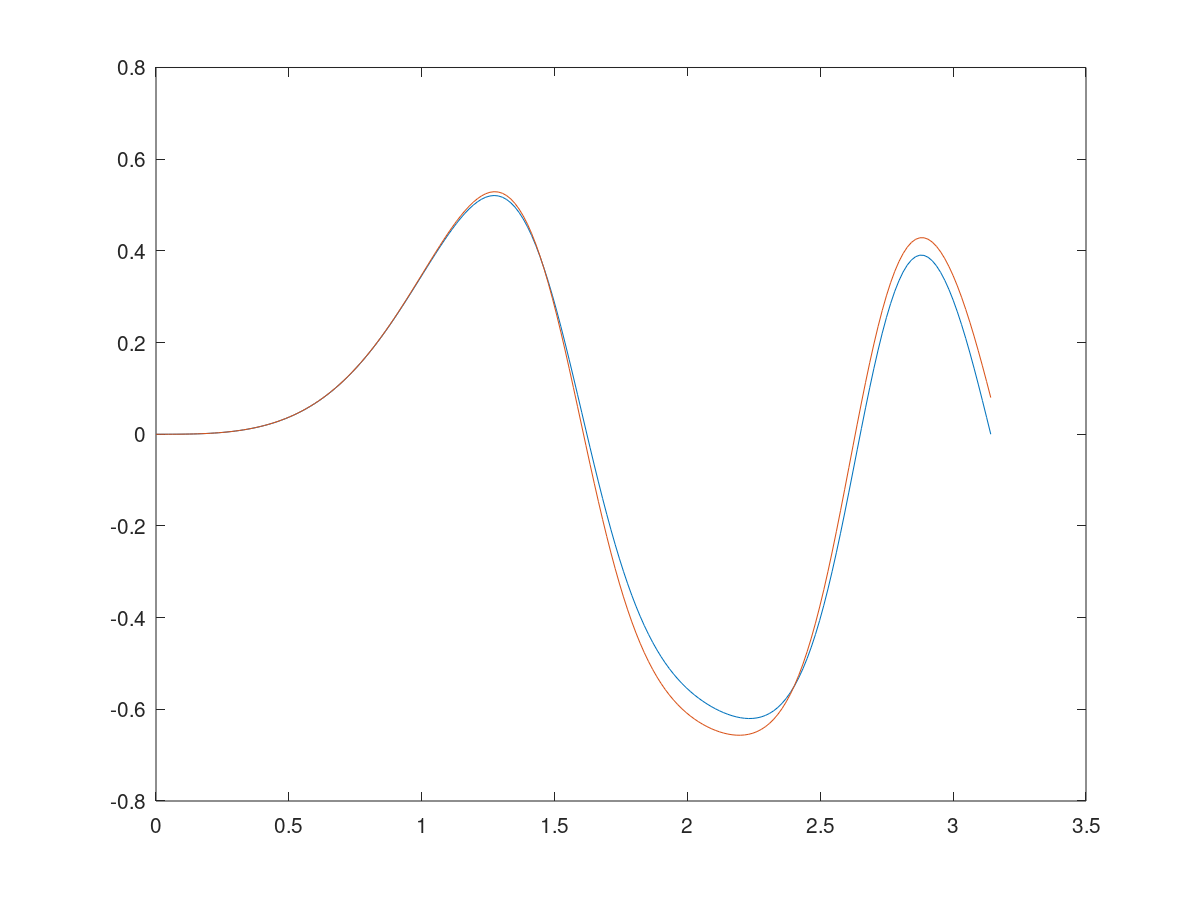}
\end{center}
\caption{Plots of $\mbox{Re}\, q_1(x)$, $\mbox{Re} \, \tilde q_1(x)$ and $\mbox{Im}\, q_1(x)$, $\mbox{Im} \, \tilde q_1(x)$ for $\de = 0.01$}
\label{fig:1}
\end{figure}

\begin{figure}[h!]
\begin{center}
\includegraphics[scale = 0.15]{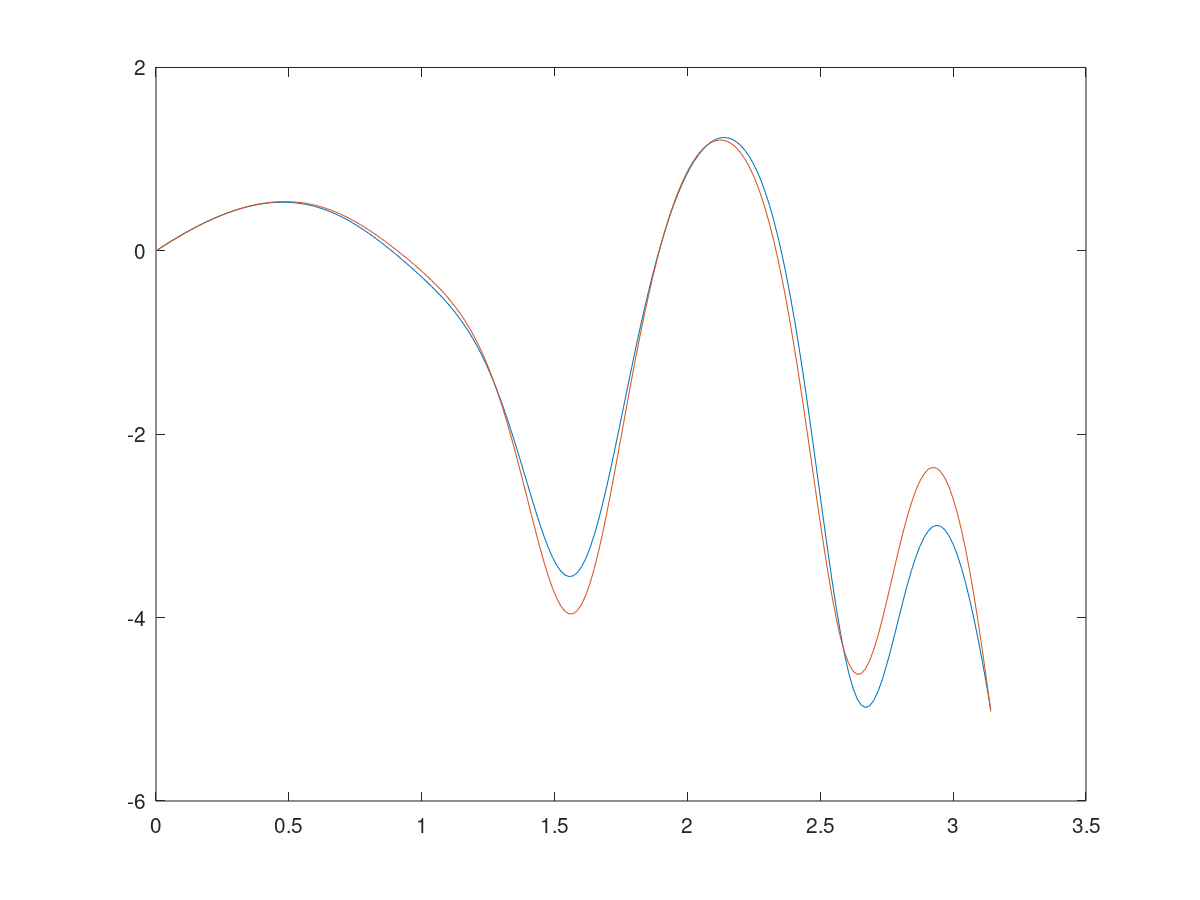}
\includegraphics[scale = 0.15]{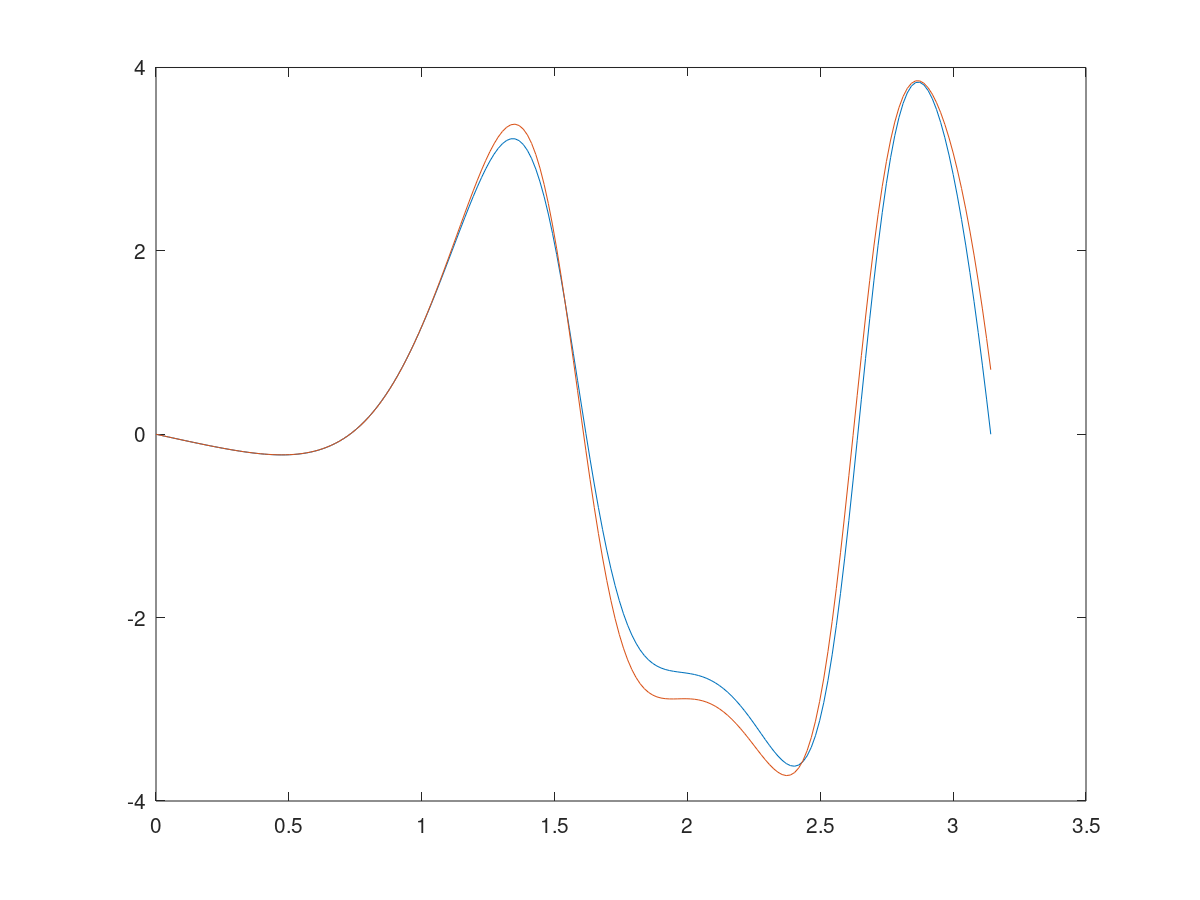}
\end{center}
\caption{Plots of $\mbox{Re}\, q_0(x)$, $\mbox{Re} \, \tilde q_0(x)$ and $\mbox{Im}\, q_0(x)$, $\mbox{Im} \, \tilde q_0(x)$ for $\de = 0.01$}
\label{fig:2}
\end{figure}

The results for different values of $\de$ are provided in the table below, where 
$$
d_1 = \max_{1 \le k \le N} |q_1(x_k) - \tilde q_1(x_k)|, \quad
d_0 = \max_{1 \le k \le N} \left| \int_0^{x_k} (q_0(t) - \tilde q_0(t)) \, dt \right|, \quad x_k = \frac{k \pi}{N}, \quad N = 200.
$$
$$
   \begin{array}{|l|l|l|l|l|l|l|}
      \hline
      \delta & d_1 & d_0 & \lambda_1 & \lambda_{-1} & M_1 & M_{-1} \\
      \hline
      0.05 & 0.4157 & 1.1131 & 0.724 & 0.276-0.200\mathrm{i} & -0.318-0.356\mathrm{i} & 0.356\mathrm{i} \\
      \hline
      0.02 & 0.1881 & 0.4805 & 0.641 & 0.359-0.080\mathrm{i} & -0.318-0.563\mathrm{i} & 0.563\mathrm{i} \\
      \hline
      0.01 & 0.0982 & 0.2463 & 0.600 & 0.400-0.040\mathrm{i} & -0.318-0.796\mathrm{i} & 0.796\mathrm{i} \\
      \hline      
      0.005 & 0.0501 & 0.1242 & 0.571 & 0.429-0.020\mathrm{i} & -0.318-1.125\mathrm{i} & 1.125\mathrm{i} \\
      \hline            
      0.002 & 0.0202 & 0.0498 & 0.545 & 0.455-0.008\mathrm{i} & -0.318-1.779\mathrm{i} & 1.779\mathrm{i} \\
      \hline    
      0.001 & 0.0101 & 0.0248 & 0.532 & 0.468-0.004\mathrm{i} & -0.318-2.516\mathrm{i} & 2.516\mathrm{i} \\
      \hline    
      0.0005 & 0.0051 & 0.0124 & 0.522 & 0.478-0.002\mathrm{i} & -0.318-3.559\mathrm{i} & 3.559\mathrm{i} \\
      \hline          
      0.0002 & 0.0020 & 0.0049 & 0.514 & 0.486-0.0008\mathrm{i} & -0.318-5.627\mathrm{i} & 5.627\mathrm{i} \\
      \hline          
      0.0001 & 0.0010 & 0.0024 & 0.510 & 0.490-0.0004\mathrm{i} & -0.318-7.958\mathrm{i} & 7.958\mathrm{i} \\
      \hline          
    \end{array}
$$

The method used for obtaining these results is based on the constructive solution of Inverse Problem~\ref{ip:M} provided in Section~\ref{sec:alg}. We use the model problem $L(0, 0)$, so the inverse problem is reduced to a finite $(4 \times 4)$ system of linear algebraic equations.

\section{Appendix}

Here we provide auxiliary lemmas about rational functions.

Denote by $\mathfrak R_N$ the class of rational functions of form $\frac{P_{N-1}(\la)}{Q_N(\la)}$, where $P_{N-1}(\la)$ is a polynomial of degree at most $(N-1)$ and $Q_N(\la)$ is a polynomial of degree $N$ with the leading coefficient equal $1$.

\begin{lem} \label{lem:rat}
Let $F(\la) := \frac{P_{N-1}(\la)}{Q_N(\la)}$ be a fixed functions of $\mathfrak R_N$ such that the zeros $\{ \la_n\}_{n = 1}^N$ of the polynomial $Q_N(\la)$ lie in $\mbox{int}\,\ga$, where $\ga := \{ \la \in \mathbb C \colon |\la| = r\}$, $r > 0$. Then there exists $\de > 0$ such that, for any function 
$$
\tilde F(\la) = \dfrac{\tilde P_{N-1}(\la)}{\tilde Q_N(\la)} \in \mathfrak R_N
$$
satisfying the estimate
\begin{equation} \label{estF}
|F(\la) - \tilde F(\la)| \le \de, \quad \la \in \ga,
\end{equation}
the zeros $\{ \tilde \la_n\}_{n = 1}^N$ of the denominator $\tilde Q_N(\la)$ also lie in $\mbox{int}\, \ga$ and
$$
|\la_n - \tilde \la_n| \le C \de^{1/m_n}, \quad n = \overline{1, N},
$$
where $m_n$ is the multiplicity of the corresponding zero $\la_n$, and the constant $C$ depends only on $F(\la)$.
\end{lem}

The proof of Lemma~\ref{lem:rat} is based on several auxiliary lemmas.

\begin{lem} \label{lem:rat1}
Let $\{ s_j\}_{j = 1}^{2N}$ be distinct points in $\ga$. Then a  function $F \in \mathfrak R_N$ is uniquely specified by its valued at these points.
\end{lem}

\begin{proof}
Suppose that, on the contrary, there exist two distinct functions
$$
F(\la) = \frac{P_{N-1}(\la)}{Q_N(\la)}, \quad \tilde F(\la) = \frac{\tilde P_{N-1}(\la)}{\tilde Q_N(\la)}, \quad F(s_j) = \tilde F(s_j), \quad j = \overline{1, 2N}.
$$
Then the polynomial
$$
P_{N-1}(\la) \tilde Q_N(\la) - Q_N(\la) \tilde P_{N-1}(\la)
$$
of degree at most $(2N - 1)$ has zeros $\{ s_j\}_{j = 1}^{2N}$. Hence, this polynomial is identically zero, so $F(\la) \equiv \tilde F(\la)$.
\end{proof}

Denote by $\{ p_k \}_{k = 0}^{N-1}$ and $\{ q_k \}_{k = 0}^{N-1}$ the coefficients of the polynomials $P_{N-1}(\la)$ and $Q_N(\la)$, respectively:
$$
	P_{N-1}(\la) = \sum_{k = 0}^{N-1} p_k \la^k, \quad Q_N(\la) = \sum_{k = 0}^{N-1} q_k \la^k + \la^N.
$$
The analogous notations $\{ \tilde p_n\}_{n = 0}^{N-1}$ and $\{\tilde q_n\}_{n = 0}^{N-1}$ will be used for the coefficients of the polynomials $\tilde P_{N-1}(\la)$ and $\tilde Q_N(\la)$, respectively.

\begin{lem} \label{lem:rat2}
Suppose that $F \in \mathfrak R_N$ fulfills the conditions of Lemma~\ref{lem:rat}. Then, there exists $\eps > 0$ such that, for any function $\tilde F$ satisfying the conditions of Lemma~\ref{lem:rat}, the following estimate holds:
\begin{equation*} 
|q_k - \tilde q_k| \le C \eps, \quad k = \overline{0, N-1},
\end{equation*}
where the constant $C$ depends only on $F$.
\end{lem}

\begin{proof}
Choose arbitrary distinct points $\{ s_j\}_{j = 1}^{2N}$ in $\ga$ and put $v_j = F(s_j)$. Consider the following system of linear algebraic equations
\begin{equation} \label{syspq}
 \sum_{k = 0}^{N-1} p_k s_j^k - v_j \left( \sum_{k = 0}^{N-1} q_k s_j^k + s_j^N\right) = 0, \quad j = \overline{1, 2N},
\end{equation}
with respect to the $2N$ unknown values $\{ p_k\}_{k = 0}^{N - 1}$ and $\{ q_k\}_{k = 0}^{N - 1}$. By virtue of Lemma~\ref{lem:rat1}, the system~\eqref{syspq} is uniquely solvable, so its determinant $\Delta$ is non-zero. The numbers $\{ \tilde p_k\}_{k = 0}^{N-1}$ and $\{ \tilde q_k\}_{k = 0}^{N-1}$ satisfy the similar system with $v_j$ replaced by $\tilde v_j = \tilde F(s_j)$. Due to \eqref{estF}, $|v_j - \tilde v_j| \le \eps$, $j = \overline{1, 2N}$. Therefore, $|\Delta - \tilde \Delta| \le C \eps$. Hence, for sufficiently small $\eps > 0$, we have $\tilde \Delta \ne 0$. Find $q_k$ from te system~\eqref{syspq}, by using Cramer's rule: $q_k = \frac{\Delta_k}{\Delta}$, $k = \overline{0, N-1}$, where $\Delta_k$ are the corresponding determinants. Clearly, $|\Delta_k - \tilde \Delta_k| \le C \eps$, $k = \overline{0, N-1}$. Hence
$$
|q_k - \tilde q_k| = \left| \frac{\Delta_k}{\Delta} - \frac{\tilde \Delta_k}{\tilde \Delta}\right| \le C \eps, \quad k = \overline{0, N-1}.
$$
\end{proof}

\begin{lem} \label{lem:rat3}
Let $\la_0$ be a zero of multiplicity $m$ of a polynomial
$$
Q_N(\la) = \sum_{k = 0}^{N-1} q_k \la^k + \la^N
$$	
Then, there exists $\eps > 0$ such that every polynomial
$$
\tilde Q_N(\la) = \sum_{k = 0}^{N-1} \tilde q_k \la^k + \la^N,
$$	
with coefficients satisfying the estimate
\begin{equation} \label{estqk}
\de := \max_{k = \overline{0, N-1}} |q_k - \tilde q_k| \le \eps,
\end{equation}
has zeros $\{ \tilde \la_j\}_{j = 1}^m$ (counting with multiplicities) satisfying the estimate
\begin{equation} \label{estlaj}
|\tilde \la_j - \la_0| \le C \de^{1/m}, \quad j = \overline{1, m},
\end{equation}
where the constant $C$ depends only on the polynomial $Q_N$.
\end{lem}

\begin{proof}
Let $\ga_j := \{ \la \in \mathbb C \colon |\la - \la_0| = r_j\}$, $j = 0, 1$, be contours not encircling other zeros of the polynomial $Q_N(\la)$ except $\la_0$, $0 < r_0 < r_1$. Using \eqref{estqk}, we obtain
$$
|\tilde Q_N(\la) - Q_N(\la)| < |Q_N(\la)|, \quad \la \in \ga_0,
$$
for sufficiently small $\eps > 0$. Applying Rouche's theorem, we conclude that $\tilde Q_N(\la)$ has exactly $m$ zeros $\{ \tilde \la_j\}_{j = 1}^m$ (counting with multiplicities) inside $\ga_0$. 
Fix $j \in \{1, \ldots, m\}$.
Note that
$$
|Q_N(\tilde \la_j)| = |Q_N(\tilde \la_j) - \tilde Q_N(\tilde \la_j)| \le C \de.
$$
On the other hand, Taylor's formula implies
\begin{equation} \label{tayQ}
Q_N(\tilde \la_j) = \sum_{k = 0}^{m-1} \frac{1}{k!} \frac{d^k}{d \la^k} Q_N(\la_0) (\tilde \la_j - \la_0)^k + \frac{d^m}{d \la^m} Q_N(\la_0)(\tilde \la_j - \la_0)^m + 
R_{m + 1}(\tilde \la_j),
\end{equation}
where
$$
R_{m + 1}(\la) := \frac{(\la - \la_0)^{m + 1}}{2 \pi i} \oint\limits_{\ga_1} \frac{Q_N(z) \, dz}{(z - \la_0)^{m + 1} (z - \la)}.
$$
It is clear that
$$
|R_{m + 1}(\la)| \le C r_0 |\la - \la_0|^m, \quad \la \in \mbox{int}\, \ga_0.
$$
Note that the radius $r_0 > 0$ can be chosen arbitrarily small by the choice of $\eps$.
Since
$$
\frac{d^k}{d \la^k} Q_N(\la_0) = 0, \quad k = \overline{0, m-1}, \qquad
\frac{d^m}{d \la^m} Q_N(\la_0) \ne 0,
$$
for significantly small $\eps > 0$, relation \eqref{tayQ} implies the estimate $|\tilde \la_j - \la_0|^m \le C \de$, which yields~\eqref{estlaj}.
\end{proof}

Lemmas~\ref{lem:rat2} and \ref{lem:rat3} together imply Lemma~\ref{lem:rat}.

\bigskip

{\bf Funding.} This work was supported by Grant 20-31-70005 of the Russian Foundation for Basic Research.

\medskip

{\bf Conflict of interest.} The authors declare that this paper has no conflict of interest.

\medskip

{\bf Authors' contributions.} N.P. Bondarenko has obtained all the theoretical results of this paper (Sections \ref{sec:intr}-\ref{sec:mult} and Appendix). A.V.~Gaidel has obtained the results of numerical experiments provided in Section~\ref{sec:ex}. The both authors have read and approved the final manuscript.

\noindent Natalia Pavlovna Bondarenko \\
1. Department of Applied Mathematics and Physics, Samara National Research University, \\
Moskovskoye Shosse 34, Samara 443086, Russia, \\
2. Department of Mechanics and Mathematics, Saratov State University, \\
Astrakhanskaya 83, Saratov 410012, Russia, \\
e-mail: {\it bondarenkonp@info.sgu.ru}

\medskip

\noindent Andrey Viktorovich Gaidel \\
1. Department of Technical Cybernetics, Samara National Research University, \\
Moskovskoye Shosse 34, Samara 443086, Russia, \\
2. Department of Mechanics and Mathematics, Saratov State University, \\
Astrakhanskaya 83, Saratov 410012, Russia, \\
3. Department of Video Mining, IPSI RAS -- Branch of the FSRC ``Crystallography and Photonics'' RAS, \\ 
Molodogvardeyskaya 151, Samara 443001, Russia. \\
e-mail: {\it andrey.gaidel@gmail.com}

\end{document}